\documentclass[11pt]{amsart}
\usepackage[dvips]{graphicx}
\usepackage{epsfig}
\usepackage{amsmath}
\usepackage{amssymb}
\usepackage[usenames, dvipsnames]{color}
\usepackage[pagebackref]{hyperref}
\usepackage{verbatim}
\usepackage[normalem]{ulem}

\theoremstyle{plain}
\newtheorem{theorem}{Theorem}[section]
\newtheorem{lemma}[theorem]{Lemma}
\newtheorem{corollary}[theorem]{Corollary}

\theoremstyle{definition}
\newtheorem{definition}[theorem]{Definition}

\theoremstyle{remark}
\newtheorem{remark}[theorem]{Remark}

\def\al{\alpha}

\def\O{\Omega}

\def\be{\begin{equation}}
\def\ee{\end{equation}}
\def\bes{\begin{equation*}}
\def\ees{\end{equation*}}
\def\bali{\begin{aligned}}
\def\eali{\end{aligned}}
\def\al{\begin{aligned}}
\def\eal{\end{aligned}}

\def\2O{\underline{\O}}

\numberwithin{equation}{section}

\makeatletter
\def\dashint{\operatorname%
{\,\,\text{\bf--}\kern-.98em\DOTSI\intop\ilimits@\!\!}}
\makeatother

\begin{document}
\title[concavity of entropy]{ concavity and other properties of the entropy on manifolds}
\author[Fu]{Xuenan Fu}
\address{School of Mathematical Sciences, East China Normal University, 500 Dongchuan Road, Shanghai 200241, P. R. of China,
E-mail address:
xuenanfu97@163.com,
}

\author[Lu]{Juanling Lu}

   \address{School of Mathematical Sciences, East China Normal University, 500 Dongchuan Road, Shanghai 200241, P. R. of China,
 E-mail address:    lujuanlingmath@163.com,}

 \author[Zhang]{Qi S. Zhang}
\address{Department of Mathematics, University of California, Riverside, CA 92521, USA, E-mail address: qizhang@math.ucr.edu.}

\date{}

\begin{abstract}
In this paper, we establish systematic estimates for the entropy and its density on Riemannian manifolds, focusing on the more challenging  cases where the Ricci curvature changes sign or the boundary is nonconvex. For example,
the refined second  law of thermodynamics states that the entropy  in a compact  domain in $\mathbb R^n$ is increasing in time, furthermore, it is concave if the domain is convex. The concavity property is equivalent to the property that the Fisher information is decreasing, which also holds for convex domains in a Riemannian manifold with nonnegative Ricci curvature (cf. \cite{NiLei}).  In view of the wide application of entropy in mathematics, information theory, physics, etc., there is certain desire in the community to extend the property to broader settings, especially to the case with nonconvex boundary (see e.g. \cite[p. 3]{CFM}), even in the Euclidean case. Here, we prove  that the refined second law still holds if the domain is not too far from convex and the negative part of the Ricci curvature is not too large, in an explicit, nonperturbative sense, thus realizing some of the expectations.
The proof is based on a new sharp second order log Poincar\'e inequality that does not require explicit curvature conditions of the manifold.  If the negative part of the Ricci curvature is too large, a counterexample to the concavity is given.
Some other related estimates for the entropy density (Hamilton type estimates) are also proven.
\end{abstract}

\maketitle
\setcounter{tocdepth}{1}
\tableofcontents

\vspace{0.5cm}
\section{\bf{Introduction}}
\subsection{Concavity of Entropy and second order log  Poincar\'e Inequalities}

Let $(\mathbf{M}, g)$ be a compact Riemannian manifold with $C^2$ boundary and $u \in C^\infty(\mathbf{M} \times (0,\infty))$ be a positive solution of the  heat equation with Neumann boundary conditions
\begin{equation}\label{eq:heat eq}
\begin{cases} 	
	(\Delta - \partial_t) u = 0 & \text{in } \mathbf{M} \times (0,\infty),\\ 	\partial_\nu u = 0 & \text{on } \partial \mathbf{M} \times (0,\infty),
\end{cases}
\end{equation}
where $\nu$ denotes the outward unit normal vector on $\partial \mathbf{M}$.
We recall the entropy
\begin{equation}\label{equ:Et}
E(t) := -\int_{\mathbf{M}} u(x,t) \ln u(x,t) \, d\mathrm{vol}_g,
\end{equation}
which is often referred to as the Boltzmann-Shannon-Nash  entropy. It is a fundamental concept in thermodynamic theory and mathematics.

Taking the time derivative of the entropy $E(t)$ yields the Fisher information
\begin{equation}\label{def:Fisher info}
\mathcal{I}(t) := \int_{\mathbf{M}} u(x,t) \, \bigl|\nabla \ln u(x,t)\bigr|^2 \, d\mathrm{vol}_g.
\end{equation}
By differentiating the Fisher information $\mathcal{I}(t)$, applying the Bochner formula, and integrating by parts, one obtains
\begin{equation}\label{eq:E''t}
\begin{aligned}  E''(t) &= -2\int_{\mathbf{M}} u\left[ |\nabla^2\ln u|^2 + \operatorname{Ric}(\nabla\ln u, \nabla\ln u) \right]\,d\mathrm{vol}_g \\  &\quad - 2\int_{\partial \mathbf{M}} u\,\mathrm{II}\left( \nabla^{\partial}\ln u, \nabla^{\partial}\ln u \right)\,dA_g,
 \end{aligned}
\end{equation}
where $\nabla^{\partial}$ denotes the gradient on $\partial \mathbf{M}$ with respect to the induced metric, and $\mathrm{II}$ is the second fundamental form defined by
\[
\mathrm{II}(X, Y) = g(\nabla_X \nu, Y)
\]
for $C^1$ vector fields $X, Y$ tangent to $\partial \mathbf{M}$. The detailed calculations can be found in \cite[p.~92]{NiLei}.

If the manifold has nonnegative Ricci curvature and the boundary is convex ($\mathrm{II}\geq0$), then $E''(t) \leq 0$, namely the entropy is concave in time. This is the result of Ni \cite{NiLei} more than two decades ago.  See also the earlier result \cite{Co} by Costa when $\mathbf M$ is the Euclidean space. In view of the wide application of entropy in mathematics, information theory, physics, etc., it is desirable to extend the property to broader settings, which is a goal of the current paper. Note that the negative part of the Ricci curvature and nonconvex part of the boundary will create integral terms with positive sign in \eqref{eq:E''t}, which cannot be eliminated easily. To overcome this difficulty, we establish a sharp second order log Poincar\'e inequality (Lemma \ref{Lem: poinca th1}).

In order to state the main results, we introduce some relevant notations and concepts. We use $\mathbf M$ or $\mathbf M^{n}$ to denote $n$-dimensional,  connected  Riemannian manifolds with metric $g$, and respectively $\Delta$, $\nabla$ and $\nabla^2$ to denote the Laplace-Beltrami operator, gradient and Hessian under $g$. Also $d(x, y)$, $d\mathrm{vol}_g$ and $dA_g$ denote the Riemannian distance between $x$ and $y$, Riemannian volume element and boundary volume element. For any point $x \in \mathbf{M}$, let $\xi(x)$ denote the smallest eigenvalue of the Ricci tensor at $x$.
The negative part of the Ricci curvature is defined as
\begin{equation*}
	\operatorname{Ric}^-(x) = (-\xi(x))_+ = \max\{0, -\xi(x)\}.
\end{equation*}
Thus, \(\operatorname{Ric}^-:\mathbf M\to[0,\infty)\) is a nonnegative function describing the negative part of the smallest eigenvalue of the Ricci tensor. In particular, for every tangent vector \(X\in T_x\mathbf M\),
\[
\operatorname{Ric}(X,X)
\geq -\operatorname{Ric}^-(x)|X|^2.
\]
Since \(\mathbf{M}\) is compact, \(\operatorname{Ric}^-\) is bounded. We set
\[
\|\operatorname{Ric}^-\|_{L^\infty}
:= \sup_{x\in \mathbf{M}} \operatorname{Ric}^-(x).
\]
Consequently, for all \(x\in \mathbf{M}\) and \(X\in T_x\mathbf{M}\),
\begin{equation}\label{def:Ric}
	\operatorname{Ric}(X,X)
	\geq -\|\operatorname{Ric}^-\|_{L^\infty} |X|^2.
\end{equation}

For a compact Riemannian manifold \(\mathbf{M}\) with \(C^{2}\) boundary \(\partial \mathbf{M}\), one considers the Neumann Laplacian, whose eigenfunctions satisfy the Neumann boundary condition \(\frac{\partial \phi}{\partial \nu} = 0\) on \(\partial \mathbf{M}\). Let \(\lambda_1^{N}(\mathbf{M}, g) > 0\) denote the first nonzero Neumann eigenvalue. By the variational principle, it is given by
\begin{equation}\label{lambda1N}
\lambda_1^N(\mathrm{\bf{M}}, g) = \inf \left\{ \frac{\int_{\mathrm{\bf{M}}} \left| \nabla \phi\right| ^2 \,d\mathrm{vol}_g}{\int_{\mathrm{\bf{M}}} \left| \phi - \bar{\phi}\right| ^2 \,d\mathrm{vol}_g} : \phi \in W^{1,2}(\mathrm{\bf{M}}), \phi \not\equiv \text{const} \right\},
\end{equation}
where $\bar{\phi} = \frac{1}{\text{Vol}(\mathrm{\bf{M}})} \int_{\mathrm{\bf{M}}} \phi \,d\mathrm{vol}_g$ is the mean value of $\phi$ over $\mathrm{\bf{M}}$. If there is no confusion or $\partial \mathbf{M} = \emptyset$, we also use $\lambda_1$ to denote $\lambda^N_1$,  i.e.,
\begin{equation}\label{lambda1}
	\lambda_1 = \inf \left\{ \frac{\int_{\mathrm{\bf{M}}} \left| \nabla \phi\right| ^2 \,d\mathrm{vol}_g}{\int_{\mathrm{\bf{M}}} \phi^2 \,d\mathrm{vol}_g} : \phi \in W^{1,2}(\mathrm{\bf{M}}), \int_{\mathrm{\bf{M}}} \phi \,d\mathrm{vol}_g = 0 \right\} .
\end{equation}

Next, for any $u \in H^1(\mathrm{\bf{M}})$, we recall the standard trace inequality in e.g. \cite[p. 1032]{Klaus Ecker},
\begin{equation}\label{trace ineq}
	\int_{\partial \mathrm{\bf{M}}} u^2\,dA_g \le C_{\mathrm{tr}}(\mathrm{\bf{M}}, g) \int_\mathrm{\bf{M}} \bigl(|\nabla u|^2+u^2\bigr)\,d\mathrm{vol}_g.
\end{equation}
The optimal constant in this inequality is given by $C_{\mathrm{tr}}(\mathrm{\bf{M}}, g) = \lambda_{tr}^{-1}$, where $\lambda_{tr}$ is defined via the Rayleigh quotient minimization
\begin{equation}\label{Def:tr const}
\lambda_{tr}=\inf\limits_{u\not\equiv 0}\, \frac{\int_\mathrm{\bf{M}}(|\nabla u|^2+u^2)\,d\mathrm{vol}_g}{\int_{\partial \mathrm{\bf{M}}}u^2\,dA_g}.
\end{equation}
Notice that $\lambda_{tr} > 0$, as the numerator is strictly positive for any nontrivial $u$.
The minimizer $u$ of this quotient satisfies the following boundary value problem
$$
\begin{cases}
	-\Delta u+u=0, & \text{in }\mathrm{\bf{M}},\\[2mm]
	\partial_\nu u=\lambda_{tr} u, & \text{on }\partial \mathrm{\bf{M}}.
\end{cases}
$$
Consequently, the optimal trace constant is precisely determined by the lowest eigenvalue of a Steklov type eigenvalue problem with a mass term.

 The following is the first main theorem.
\begin{theorem}\label{th: concave1}
	Assume $(\mathbf{M}, g)$ is a compact Riemannian manifold with dimension $n$.
	\begin{itemize}
		\item[(a)] Suppose $\mathbf{M}$ is closed (i.e., $\partial \mathbf{M} = \emptyset$) and 	\[
		\|\operatorname{Ric}^{-}\|_{L^{\infty}} \leq \frac{\lambda_1(\mathbf{M}, g)}{n}.
		\]
		Let $u=u(x, t)$ be a positive solution of the  heat equation $\Delta u - \partial_t u = 0$ on $\mathbf{M}\times(0,\infty)$. Then $E''(t) \leq 0$ for all $t > 0$.
		\item[(b)] More generally, suppose $\mathbf{M}$ has $C^2$ boundary $\partial \mathbf{M}$ satisfying $\mathrm{II} \ge -\sigma g_{\partial\mathbf{M}}$ for a constant $\sigma \ge 0$
		and
		$$
		\|\operatorname{Ric}^{-}\|_{L^{\infty}} \leq \frac{\lambda_1^N(\mathbf{M}, g)}{n}
		-\frac{\sigma\lambda_1^N(\mathbf{M}, g)(n+4\sqrt n+8)}{2n\lambda_{tr}}-\frac{\sigma}{\lambda_{tr}}.
		$$
		 Let $u=u(x, t)$ be a positive solution of  \eqref{eq:heat eq}. Then $E''(t) \leq 0$ for all $t > 0$.
	\end{itemize}

	In either case, the entropy is concave in time, implying the Fisher information is nonincreasing.

\end{theorem}

Recall that Yang showed the following explicit estimate in \cite{Yang90} for the first eigenvalue on $\mathbf{M}$:
\begin{equation}\label{negative}
\lambda_1 \geq \frac{\pi^2}{D^2}e^{-C_n\sqrt{K}D},
\end{equation}
where $C_n = \max\{\sqrt{n-1},\sqrt{2}\}$, $D$ is the diameter of $\mathbf{M}$ and $K=\Vert Ric^-\Vert_\infty$.
Combining this estimate with part (a) of the theorem, we immediately obtain:

\begin{corollary}

Let $\mathbf{M}$ and $u$ be as in the theorem and $D$ be the diameter of $\mathbf{M}$.  If 
\[
D^2 \Vert Ric^-\Vert_\infty e^{\left( \max\{\sqrt{n-1}, \sqrt{2} \} D \sqrt{\Vert Ric^-\Vert_\infty} \right)} \le \frac{\pi^2}{n},
\]then, 
\[
E''(t) \le 0, \quad t \ge 0.
\]
\end{corollary}

Note the left hand side is scaling invariant and can be realized easily by perturbing compact Ricci flat manifolds.

	\begin{remark} 

Condition (b) in Theorem~\ref{th: concave1} implicitly implies that the nonconvexity parameter $\sigma$ is small enough, i.e.,
	\[
		\sigma\le
		\frac{2\lambda_1^N(\mathbf{M}, g)\lambda_{\mathrm{tr}}}
		{\lambda_1^N(\mathbf{M}, g)(n+4\sqrt{n}+8)+2n}.
		\]
Thus, entropy concavity remains valid when the boundary is sufficiently close to convex and the negative part of the Ricci curvature is suitably bounded.
	\end{remark}
	
	\begin{remark}\label{counter eg}
	As mentioned in the abstract, this result addresses an expectation in the community. See also Theorem \ref{noncpt E concave} for a result for noncompact manifolds. Moreover, the following example shows that entropy concavity may fail when \(|\operatorname{Ric}^{-}|\) is sufficiently large. Thus, some restriction on the negative Ricci curvature is necessary in general.
	
	Let $(\mathbf{N}^3,h)$ be a 3-dimensional closed hyperbolic manifold with constant sectional curvature -1. Then the 4-dimensional warped product manifold
	\[
 (\mathbf{M}^4,g)=\left(\mathbb{R}\times\mathbf{N}^3, ds^2+\cosh^2(s)h\right)
	\]
has constant sectional curvature -1. Hence
	\[
	\operatorname{Ric}_\mathbf{M}=-3g.
	\]
	Let $u$ be a positive solution to the heat equation $\partial_{t} u=\Delta_{g} u$ with $u_0$ as the initial value, i.e.,
	\[
	u(x,t)=\int_{\mathbf{M}} G(x,t,y)u_0(y) d\mathrm{vol}_g(y),
	\]
	where $G(x, t, y)$ is the heat kernel with pole at $y$.
	
	Define the function
	\[
	u_0=\frac{1}{\zeta}\cosh^{-\frac{7}{2}}s
	\]
	on $\mathbf{M}$, where
	\[
	\zeta=\operatorname{Vol}(N)\int_{\mathbb{R}}\cosh^{-\frac{1}{2}}s\,ds
	<\infty.
	\]
	Denote
	\[
	f_0=\ln u_0=\ln(\frac{1}{\zeta})-\frac{7}{2}\ln\cosh s.
	\]
	So we have
	\[
	\nabla f_0=-\frac{7}{2}\tanh s\,\partial_s.
	\]
	Thus we obtain
	\[
	\operatorname{Ric}\left(\nabla f_0,\nabla f_0\right)
	=-3g(-\frac{7}{2}\tanh s\,\partial_s, -\frac{7}{2}\tanh s\,\partial_s)
	=-\frac{147}{4}\tanh^2 s,
	\]
	\[
	|\nabla^2 f_0|^2=g^{ik}g^{jl}(f_0)_{ij}(f_0)_{kl}
	=\frac{49}{4}\left(\frac{1}{\cosh^{4} s}+3\tanh^4 s\right).
	\]
	This shows
	\[
	|\nabla^2 f_0|^2+\operatorname{Ric}\left(\nabla f_0,\nabla f_0\right)
	=\frac{49}{4}\frac{1}{\cosh^2 s}\,(1-4\tanh^2 s).
	\]
	Therefore
	\begin{align*}
		\int_M u_0\left(|\nabla^2 f_0|^2+\operatorname{Ric}\left(\nabla f_0,\nabla f_0\right)\right)d\mathrm{vol}_g
		&=\frac{49}{4}\frac{1}{\zeta}\int_{\mathbb{R}} \int_{\mathbf{N}} \cosh^{-\frac{5}{2}}s(1-4\tanh^2 s) d\mathrm{vol}_h ds=-\frac{7}{12},
	\end{align*}
	where we have used Fubini's theorem and $d\mathrm{vol}_g=\cosh^3 s~ds d\mathrm{vol}_h$.
	Substituting the above equality into the expression for $E''(t)$ in \eqref{eq:E''t} yields
	\[
	E''(0^+)=-2\left(-\frac{7}{12}\right)=\frac{7}{6}>0.
	\]
	
	Since the initial value $u_0$ is smooth, strictly positive, and $E''(t)$ is continuous near $t=0$, then there exists $\varepsilon > 0$ such that
	\[
	E''(t)>0,\qquad 0<t<\varepsilon.
	\]
	
	\end{remark}
	

\subsection{Density of Entropy (Hamilton Type Estimates) on Manifolds with Nonconvex Boundaries under Pointwise Ricci Curvature Bounds}

In 1993, R. Hamilton \cite{Hamilton} derived the following log-gradient estimate, which is essentially an estimate on entropy density.

\vspace{.3cm}
{\bf{Theorem}} (\cite{Hamilton}) {\it{Let $(\mathbf{M}, g)$ be a closed manifold whose Ricci curvature is bounded from below by $-K \le 0$. If $u=u(x, t)\leq A\in\mathbb R^1$ is a positive solution to the heat equation on $\mathbf{M}\times(0, \infty)$, then
$$t \frac{\left| \nabla u\right| ^2}{u} \le (1+2Kt)u \ln \frac{A}{u}.$$}}

On a compact manifold with smooth convex boundary, the proof is the same as in \cite{Hamilton}. When the boundary is not convex, Hamilton's estimate has not been  established yet, although related gradient estimates without the log functions have appeared recently in \cite{Sturm}. The second topic of the paper is to address this question. The idea is to use certain auxiliary functions to cancel out the bad terms caused by the boundary. To extend the estimate to nonconvex boundaries, we assume that $\partial \mathbf{M}$
satisfies the interior rolling $R$-ball condition.
\begin{definition}\label{Def: inte roll R}
	The boundary $\partial \mathrm{\bf{M}}$ is said to satisfy the interior rolling $R$-ball condition if for every point
	$p \in \partial \mathrm{\bf{M}}$ there exists a point
	$q \in \mathrm{\bf{M}}$ such that the geodesic ball $B_q(R/2)$ of radius $R/2$ centered at $q$ satisfies
	\[
B_q(R/2) \cap \partial \mathrm{\bf{M}}=	\{p\} \quad \text{and} \quad B_q(R/2) \subset \mathrm{\bf{M}}.
	\]
	
	In other words, at each boundary point one can place a geodesic ball of radius $R/2$ entirely inside $\mathrm{\bf{M}}$ that touches the boundary exactly at $p$.
\end{definition}
The next result extends  Hamilton's entropy density (gradient) estimate to manifolds with nonconvex boundary.

\begin{theorem}\label{Th: ham estimate}
	Let $(\mathrm{\bf{M}}^n, g)$ be an n-dimensional compact Riemannian manifold with $C^2$ boundary $\partial \mathbf M$, and suppose that $\partial \mathbf M$ satisfies the ``interior rolling $R$-ball" condition for some $R\in(0, 1]$. Assume that the Ricci curvature and the second fundamental form satisfy \[
	 \operatorname{Ric}_\mathrm{\bf{M}}\ge -K g
\qquad\text{and}\qquad
	 \mathrm{II}_{\partial \mathrm{\bf{M}}}\ge -L g_{\partial \mathrm{\bf{M}}},
	\]
	for some constants $K,L\ge0$. Let $u=u(x,t)>0$ be a bounded smooth solution of \eqref{eq:heat eq} on $\mathbf{M}\times(0,1]$. Define $
	A:=e\sup_{\mathbf{M}\times(0,1]}u.
	$ Then there exists a constant $\alpha>0$
	such that
	\[
	\frac{t |\nabla u|^2}{u}\le \alpha u \ln \frac{A}{u}
	\]
	on $\mathbf{M}\times(0,1]$, where $\alpha=\alpha(\operatorname{diam}(\mathrm{\bf{M}}), n, K, L, R)>0$.
\end{theorem}
\subsection{Gradient Bound and Fisher Information under Integral Curvature Condition}
All of the preceding results have been established under pointwise Ricci curvature bounds. In this final part, we obtain a bound of the Fisher information under an integral Ricci curvature bound together with a volume noncollapsing condition. This is based on Wang's gradient estimate \eqref{cite: wang1} for $p=2$. Recall that Wang \cite{Wang FengYu} established a useful gradient estimate: if the Riemannian manifold has pointwise Ricci lower bound i.e., \(\operatorname{Ric} \ge -K\), $K\geq 0$,  then the following estimate
\begin{equation}\label{cite: wang1}
	|\nabla P_t f|^p \le e^{pKt} P_t (|\nabla f|^p),\qquad \forall t \ge 0
\end{equation}
holds for all \(f \in C^1(M)\), \(p\in[1, \infty)\), where $P_t$ denotes the heat semigroup, see \cite[Cor. 3.2.6]{Wang FengYu}.

 There has been interest in extending many classical results from pointwise curvature conditions to integral curvature conditions. For example, Petersen and Wei \cite{Petersen and Wei} extended the classical volume comparison theorem. Recently, the authors of \cite{Zhang and Zhu} established Li-Yau type gradient bounds on manifolds for which the negative part of the Ricci curvature satisfies $|\operatorname{Ric}^-|\in L^p$ for $p > n/2$, under a volume noncollapsing condition, see also \cite{X. Ramos}. We will use techniques developed in their work.

\begin{theorem}\label{Integral GE}
	Let $(\mathbf{M}, g)$ be an $n$-dimensional compact Riemannian manifold without boundary satisfying the following conditions:
	\begin{itemize}
		\item[(a)] $\displaystyle \int_{\mathbf{M}} |\operatorname{Ric}^-|^p \, d\mathrm{vol}_g = \sigma < \infty$, where $p > n/2$ and $\operatorname{Ric}^-$ denotes the negative part of the Ricci curvature;
		\item[(b)] noncollapsing condition: $|B(x,r)| \ge \rho \, r^n$ for all $x\in\mathbf{M}$, $0 < r \le 1$ and some constant $\rho > 0$.
	\end{itemize}
	For any $f \in C^{\infty}(\mathbf{M})$, define
	\[
	u(x,t) := \int_{\mathbf{M}} G(x, t, y) \, f(y) \, d\mathrm{vol}_g(y), \qquad
	w(x,t) := \int_{\mathbf{M}} G(x, t, y) \, |\nabla f|^2(y) \, d\mathrm{vol}_g(y),
	\]
	where $G(x, t, y)$ is the heat kernel with pole at $y$. Then the estimate
	\[
	|\nabla u|^2 \le \frac{1}{\underline{J}(t)} \, w
	\]
	holds, or equivalently,
	\[
	\left|\nabla_x \int_{\mathbf{M}} G(x, t, y) \, f(y) \, d\mathrm{vol}_g(y)\right|^2
	\le \frac{1}{\underline{J}(t)} \int_{\mathbf{M}} G(x, t, y) \, |\nabla f|^2(y) \, d\mathrm{vol}_g(y), \qquad \forall t>0.
	\]
	Here $\underline{J}(t)$ is given explicitly by
	\[
	\underline{J}(t) = 2^{-\frac{1}{a-1}} \,
	\exp\!\left( -(a-1)^\frac{n}{2p-n}\,
	\bigl[4(\sigma \hat{C}(t))^{1/p}\bigr]^{\frac{2p}{2p-n}} \, t \right),
	\]
	where $a = 5\delta^{-1}$, $0 < \delta \le \frac{5}{2}$, and $\hat{C}(t)$ is a monotonically increasing function appearing in the
	Gaussian upper bound
	$$
	G(x, t, y)\leq \frac{\hat{C}(t)}{t^{n/2}}e^{-\bar{c}d^2(x, y)/t},\qquad t\in(0, \infty),
	$$
	for the heat kernel under the above integral Ricci curvature bound and
	volume noncollapsing condition.
	Here $\bar{c}>0$ is an absolute constant.

\end{theorem}

Applying the preceding gradient estimate together with the definition of the Fisher information \eqref{def:Fisher info} yields the following estimate.
\begin{theorem}\label{th:Fisher info}
	Let $(\mathbf{M}^n, g)$ satisfy the assumptions of Theorem~\ref{Integral GE}. For a strictly positive \(u_0\in C^\infty(\mathbf{M})\), let $$u=u(x,t)=\int_\mathbf{M} G(x,t,y)u_0(y)\,d\mathrm{vol}_g(y)$$ be the positive solution of
	the heat equation $\Delta u - \partial_t u = 0$ on $\mathbf{M}\times(0,\infty)$ with initial value \(u(\cdot,0)=u_0\). Let $\mathcal I(t)$ be the corresponding Fisher information defined by \eqref{def:Fisher info}. Then for all $t > 0$,
	\[
	\mathcal{I}(t) \leq \frac{1}{\underline{J}(t)}\,\mathcal{I}(0),
	\]
	where $\underline{J}(t)$ is defined in Theorem~\ref{Integral GE}.
\end{theorem}
 However, we have not been able to prove the monotonicity of Fisher information, namely the concavity of the entropy in this integral curvature setting.

The rest of the paper is organized as follows. In Section 2, we prove the sharp second order log  Poincar\'e inequalities and Theorem \ref{th: concave1}. In Section 3, we derive the Hamilton type gradient estimates for positive solutions to \eqref{eq:heat eq} on compact manifolds with nonconvex boundaries i.e., Theorem \ref{Th: ham estimate}. Finally, in Section 4,  we establish a Wang gradient estimate for heat kernel convolutions under integral Ricci curvature bounds i.e., Theorem \ref{Integral GE}, which leads directly to the bound on the Fisher information, namely Theorem \ref{th:Fisher info}.

Monotonicity of the Fisher information continues to play important roles for the study of heat type equations and beyond. See, for examples, \cite{CFHS} and a recent result on the Boltzmann equation in \cite{ISV}.  Some of the ideas in this paper seems expandable to certain  nonlinear heat type equations, which we hope to address in the future.

\section{\bf{Concavity of Entropy and second order log Poincar\'e Inequality}}

\subsection{Compact Case}
In this section, we first prove a sharp inequality which is called a {\it{second order log  Poincar\'e inequality}} on compact Riemannian manifolds, with or without boundary.  The proof is partly motivated by some calculations in \cite{CW}.

\begin{lemma}\label{Lem: poinca th1}
	Suppose $(\mathbf{M}^n,g)$ is a compact, connected Riemannian manifold, and let $u \in C^2(\mathbf{M})$ be a positive function.
	\begin{itemize}
		\item[(a)] Suppose $\mathbf{M}$ is closed i.e., $\partial \mathbf{M} = \emptyset$. Then
		\begin{equation}\label{ln poin1}
			\int_\mathbf{M} |\nabla \ln u|^2\,u\;d\mathrm{vol}_g \;\le\; \frac{n}{\lambda_1(\mathbf{M}, g)}\int_\mathbf{M} |\nabla^2 \ln u|^2\,u\;d\mathrm{vol}_g.
		\end{equation}
		\item[(b)] More generally, suppose $\mathbf{M}$ has  $C^2$ boundary $\partial \mathbf{M}$. If $u \in C^2(\mathbf{M})$ satisfies the Neumann boundary condition $\partial_\nu u = 0$ on $\partial \mathbf{M}$, then
		\begin{equation}\label{ineq:neu poin}
			\int_\mathbf{M}|\nabla\ln u|^2 \, u \, d\mathrm{vol}_g
			\le
			\frac{n}
			{\lambda_1^N(\mathbf{M},g)}
			\int_\mathbf{M}|\nabla^2\ln u|^2 \, u \, d\mathrm{vol}_g.
		\end{equation}
	\end{itemize}
\end{lemma}

\begin{proof}[Proof of Lemma \ref{Lem: poinca th1}]
	
	We divide the proof into two cases, corresponding to the assumptions (a) and (b).\\
	{\bf{Case (a):}}
	Set
	$$
	v = \sqrt{u}.
	$$
	Since \(u\in C^2(\mathbf M)\) is positive, both \(\ln u\) and \(v=\sqrt{u}\) belong to \(C^2(\mathbf M)\). Differentiation yields
	$$
	\nabla v = \frac12\sqrt{u}\,\nabla \ln u,$$
	which gives the pointwise identity
	\begin{equation}\label{identity}
		u\left| \nabla \ln u\right| ^2 = 4\left| \nabla v\right| ^2.
	\end{equation}
	
	Let $
	\bar{v}= \frac{1}{\operatorname{Vol}_g(\mathbf{M})}\int_{\mathbf{M}}v\,d\mathrm{vol}_g$ denote the spatial average of $v$ over $\mathbf{M}$. Since $\mathbf{M}$ is a closed manifold, integration by parts gives
	$$
	\int_\mathbf{M} \left| \nabla v\right| ^2 \, d\mathrm{vol}_g = \int_\mathbf{M} \left\langle \nabla v, \nabla(v-\overline v)\right\rangle d\mathrm{vol}_g = -\int_\mathbf{M} (v-\overline v)\Delta v \, d\mathrm{vol}_g.
	$$
	Applying Cauchy-Schwarz inequality and the classical Poincar\'e inequality to $v-\overline v$ on $(\mathbf{M}, g)$, we obtain
	$$
	\begin{aligned}
		\int_\mathbf{M} \left| \nabla v\right| ^2 \, d\mathrm{vol}_g &\le \left( \int_\mathbf{M} (v-\overline v)^2 \, d\mathrm{vol}_g \right)^{1/2} \left( \int_\mathbf{M} (\Delta v)^2 \, d\mathrm{vol}_g \right)^{1/2} \\ &\le \frac{1}{\sqrt{\lambda_1(\mathbf{M}, g)}} \left( \int_\mathbf{M} \left| \nabla v\right| ^2 \, d\mathrm{vol}_g \right)^{1/2} \left( \int_\mathbf{M} (\Delta v)^2 \, d\mathrm{vol}_g \right)^{1/2}.
	\end{aligned}
	$$
	
	Since $\mathbf{M}$ is connected, if $\nabla v \equiv 0$, then $v$ is constant, and therefore $u = v^2$ is constant as well. Consequently, both sides of the asserted inequality vanish, and the result follows trivially. Otherwise, dividing both sides by $\left( \int_\mathbf{M} \left| \nabla v\right| ^2 d\mathrm{vol}_g \right)^{1/2}$ and squaring yields
	\begin{equation}\label{ineq:nabla v}
	\int_{\mathbf{M}}|\nabla v|^2 \,d\mathrm{vol}_g
	\le
	\lambda_1(\mathbf{M},g)^{-1} \int_{\mathbf{M}}(\Delta v)^2\,d\mathrm{vol}_g.
	\end{equation}

Next, we want to establish the following inequality for the Hessian of $\ln u$:
	\begin{equation}\label{inequality 1}
		\frac1n
		\int_\mathbf{M}(\Delta v)^2 \,d\mathrm{vol}_g \le \frac14 \int_\mathbf{M} u|\nabla^2\ln u|^2 \,d\mathrm{vol}_g,
	\end{equation}
which will infer the lemma via \eqref{identity} and \eqref{ineq:nabla v}.

A direct computation gives
\begin{align*}
	\nabla^2\ln u =\frac{2}{v} \left( \nabla^2v-\frac{dv\otimes dv}{v} \right).
\end{align*}
Consequently,
\begin{align}\label{equality 1}
	u|\nabla^2\ln u|^2 = 4\left(|\nabla^2v|^2\underbrace{ - 2\left\langle \nabla^2v,\frac{dv\otimes dv}{v} \right\rangle
	+ \frac{|\nabla v|^4}{v^2}}
	\right).
\end{align}

Since
$$
\nabla\left(\frac{|\nabla v|^2}{v}\right)
= \frac{1}{v}\nabla|\nabla v|^2 - |\nabla v|^2\frac{\nabla v}{v^2}
$$
and
$$
\nabla|\nabla v|^2
=
2\nabla^2v(\nabla v,\cdot),
$$
we obtain
\begin{align*}
	\left\langle \nabla\left(\frac{|\nabla v|^2}{v}\right), \nabla v \right\rangle
	= 2\left\langle \nabla^2v,\frac{dv\otimes dv}{v} \right\rangle - \frac{|\nabla v|^4}{v^2}.
\end{align*}
Hence,
	\begin{equation}\label{equality: inner product}
	\begin{aligned}
	\operatorname{div} \left(\frac{|\nabla v|^2}{v}\nabla v \right)
	&= \frac{|\nabla v|^2}{v}\Delta v
	+ \left\langle \nabla\left(\frac{|\nabla v|^2}{v}\right), \nabla v \right\rangle\\
	&=\frac{|\nabla v|^2}{v}\Delta v
	+ \underbrace{2\left\langle \nabla^2v,\frac{dv\otimes dv}{v} \right\rangle - \frac{|\nabla v|^4}{v^2}}.
\end{aligned}
\end{equation}
Substituting \eqref{equality: inner product} into \eqref{equality 1} yields
	\begin{align}\label{equality 2}
		u|\nabla^2\ln u|^2 = 4\left[|\nabla^2v|^2 + \underbrace{\frac{|\nabla v|^2}{v}\Delta v}
		- \operatorname{div} \left(\frac{|\nabla v|^2}{v}\nabla v \right) \right].
	\end{align}

In order to treat the under-braced term, we decompose
	\[
\nabla^2v = \left(\nabla^2v-\frac{\Delta v}{n}g \right) + \frac{\Delta v}{n}g
\]
and
\[
\frac{dv\otimes dv}{v} = \left(\frac{dv\otimes dv}{v} - \frac{1}{n}\frac{|\nabla v|^2}{v}g \right)
+ \frac{1}{n}\frac{|\nabla v|^2}{v}g.
\]
Since both $ \nabla^2v-\frac{\Delta v}{n}g$ and $\frac{dv\otimes dv}{v} - \frac{1}{n}\frac{|\nabla v|^2}{v}g$ are trace-free, we have
\begin{equation}\label{eq:cross}
	\left\langle \nabla^2v,\frac{dv\otimes dv}{v} \right\rangle
	= \left\langle \nabla^2v-\frac{\Delta v}{n}g, \frac{dv\otimes dv}{v}
	- \frac{1}{n}\frac{|\nabla v|^2}{v}g \right\rangle
	+ \frac{\Delta v}{n}\frac{|\nabla v|^2}{v}.
\end{equation}

	Using the trace-free decomposition above, \eqref{equality: inner product} becomes
	\begin{equation}\label{eq:traceless div}
		\operatorname{div} \left(\frac{|\nabla v|^2}{v}\nabla v \right)
		= \frac{n+2}{n}\frac{|\nabla v|^2}{v}\Delta v
		+2\left\langle \nabla^2v-\frac{\Delta v}{n}g, \frac{dv\otimes dv}{v}
		- \frac{1}{n}\frac{|\nabla v|^2}{v}g \right\rangle- \frac{|\nabla v|^4}{v^2}.
	\end{equation}
Equivalently,
\begin{equation}\label{eq:laplacian v}
	\begin{aligned}
	\frac{n+2}{n}
	\underbrace{\frac{|\nabla v|^2}{v}\Delta v}
	&=
	\frac{|\nabla v|^4}{v^2}
	+
	\operatorname{div}
	\left(
	\frac{|\nabla v|^2}{v}\nabla v
	\right)
	\\
	&\quad
	-
	2\left\langle
	\nabla^2v-\frac{\Delta v}{n}g,
	\frac{dv\otimes dv}{v}
	-
	\frac{1}{n}\frac{|\nabla v|^2}{v}g
	\right\rangle.
\end{aligned}
\end{equation}
	Substituting  \eqref{eq:laplacian v} into \eqref{equality 2}, we get
	\begin{align*}
		u|\nabla^2\ln u|^2
		=
		4\Bigg[
		&|\nabla^2v|^2
		+
		\frac{n}{n+2}
		\Bigg(
		\frac{|\nabla v|^4}{v^2}
		+
		\operatorname{div}
		\left(
		\frac{|\nabla v|^2}{v}\nabla v
		\right)
		\\
		&\qquad
		-
		2\left\langle
		\nabla^2v-\frac{\Delta v}{n}g,
		\frac{dv\otimes dv}{v}
		-
		\frac{1}{n}\frac{|\nabla v|^2}{v}g
		\right\rangle
		\Bigg)
		-
		\operatorname{div}
		\left(
		\frac{|\nabla v|^2}{v}\nabla v
		\right)
		\Bigg].
	\end{align*}
	Applying
	\[
	|\nabla^2v|^2
	=
	\frac{1}{n}(\Delta v)^2
	+
	\left|
	\nabla^2v-\frac{\Delta v}{n}g
	\right|^2,
	\]
	we deduce, by completing the square, that
	\begin{align*}
		\frac14 u|\nabla^2\ln u|^2
		=&
		\frac1n(\Delta v)^2
		+
		\left|
		\nabla^2v-\frac{\Delta v}{n}g
		-
		\frac{n}{n+2}
		\left(
		\frac{dv\otimes dv}{v}
		-
		\frac1n\frac{|\nabla v|^2}{v}g
		\right)
		\right|^2
		\\
		&-
		\left(\frac{n}{n+2}\right)^2
		\left|
		\frac{dv\otimes dv}{v}
		-
		\frac1n\frac{|\nabla v|^2}{v}g
		\right|^2
		\\
		&+
		\frac{n}{n+2}\frac{|\nabla v|^4}{v^2}
		-
		\frac{2}{n+2}
		\operatorname{div}
		\left(
		\frac{|\nabla v|^2}{v}\nabla v
		\right).
	\end{align*}
	Since
	\begin{equation*}
		\left|
		\frac{dv\otimes dv}{v}
		-
		\frac1n\frac{|\nabla v|^2}{v}g
		\right|^2
		=
		\frac{|\nabla v|^4}{v^2}
		-
		\frac{2}{n}\frac{|\nabla v|^4}{v^2}
		+
		\frac{1}{n}\frac{|\nabla v|^4}{v^2}=
		\frac{n-1}{n}\frac{|\nabla v|^4}{v^2}
	\end{equation*}
	and
	\begin{align*}
		-
		\left(\frac{n}{n+2}\right)^2
		\frac{n-1}{n}
		+
		\frac{n}{n+2}
		=
		\frac{3n}{(n+2)^2},
	\end{align*}
	we have
	\begin{align*}
		\frac14 u|\nabla^2\ln u|^2
		={}&
		\frac1n(\Delta v)^2+\frac{3n}{(n+2)^2}
		\frac{|\nabla v|^4}{v^2}
		-
		\frac{2}{n+2}
		\operatorname{div}
		\left(
		\frac{|\nabla v|^2}{v}\nabla v
		\right)
		\\
		&+
		\left|
		\nabla^2v-\frac{\Delta v}{n}g
		-
		\frac{n}{n+2}
		\left(
		\frac{dv\otimes dv}{v}
		-
		\frac1n\frac{|\nabla v|^2}{v}g
		\right)
		\right|^2
		\\
		\geq{}&
		\frac1n(\Delta v)^2
		-
		\frac{2}{n+2}
		\operatorname{div}
		\left(
		\frac{|\nabla v|^2}{v}\nabla v
		\right).
	\end{align*}

	Since $\mathbf{M}$ is closed, integrating this pointwise estimate over $\mathbf{M}$ and using the divergence theorem, we obtain \eqref{inequality 1}, namely
	\begin{equation*}
		\frac14
		\int_\mathbf{M} u|\nabla^2\ln u|^2 \,d\mathrm{vol}_g
		\geq
		\frac1n
		\int_\mathbf{M}(\Delta v)^2 \,d\mathrm{vol}_g.
	\end{equation*}
	Substituting this into \eqref{ineq:nabla v} and using \eqref{identity} yields
	\begin{equation*}
			\int_\mathbf{M} |\nabla \ln u|^2\,u\;d\mathrm{vol}_g \;\le\; \frac{n}{\lambda_1(\mathbf{M}, g)}\int_\mathbf{M} |\nabla^2 \ln u|^2\,u\;d\mathrm{vol}_g.
	\end{equation*}  The proof of Case (a) is done.

	{\bf{Case (b):}}
	The proof proceeds in an analogous manner to case (a), with the same notation.
	Indeed, setting \(f=\ln v\) and \(v=\sqrt u\), the Neumann condition yields
	\[
	\partial_\nu v=\frac{\partial_\nu u}{2\sqrt u}=0,
	\qquad
	\partial_\nu f=\frac{\partial_\nu v}{v}=0
	\quad\text{on } \partial \mathbf{M}.
	\]
	Applying the divergence theorem to \((v-\overline v)\nabla v\), we obtain
	\[
	\int_\mathbf{M} |\nabla v|^2 \, d\mathrm{vol}_g
	=
	-\int_\mathbf{M} (v-\overline v)\Delta v \, d\mathrm{vol}_g
	+
	\int_{\partial \mathbf{M}} (v-\overline v)\partial_\nu v \, dA_g,
	\]
	and the boundary integral vanishes since \(\partial_\nu v=0\) on \(\partial \mathbf{M}\). Similarly, the divergence theorem to the vector field \(X=|\nabla v|^2\nabla f\) gives
	\[
	\int_{\partial \mathbf{M}} |\nabla v|^2\partial_\nu f \, dA_g
	=
	\int_{\partial \mathbf{M}}
	\left\langle |\nabla v|^2\nabla f,\nu \right\rangle dA_g
	=
	0,
	\]
	because \(\partial_\nu f=0\) on \(\partial \mathbf{M}\). Hence both boundary terms arising in the preceding proof disappear, and the same argument applies with \(\lambda_1(\mathbf{M}, g)\) replaced by \(\lambda_1^N(\mathbf{M}, g)\).
	The proof is complete.
	
\end{proof}
We now show that the coefficient on the right hand side of \eqref{ln poin1} is sharp by the following example on the standard sphere $\mathbb{S}^n$. In fact, for a quick check in the one dimensional case, one can just take $\mathbf{M}=\mathbb{S}^1$ and use the functions $u=1+\epsilon \sin \theta$ with $\epsilon \to 0$ in \eqref{ln poin1}.
\begin{remark}
Let
\[
\mathbf{M}=\mathbb{S}^n=\left\{p=(x_1,\dots,x_{n+1})\in\mathbb{R}^{n+1}:
x_1^2+x_2^2+\cdots+x_{n+1}^2=1\right\},
\]
be the unit \(n\)-sphere centered at the origin \(O\) of $\mathbb{R}^{n+1}$, equipped with the standard round metric \(g\), and denote its volume element by \(d\sigma\).

For each \(p\in \mathbb{S}^n\), let
\[
X(p)=\overrightarrow{Op}\in\mathbb{R}^{n+1}
\]
be the position vector of \(p\). The vector \(X(p)\) has the same coordinates as \(p\), \(|X(p)|=1\), and the tangent space satisfies
\[
T_p\mathbb{S}^n=\left\{Y\in\mathbb{R}^{n+1}:\langle Y,X(p)\rangle=0\right\}.
\]
Thus \(\nu(p)=X(p)\) is the unit outward normal at \(p\).

Let
\[
E_1=(1,0,\dots,0)\in\mathbb{R}^{n+1}
\]
be a fixed constant vector, and define the coordinate function
\[
\varphi:\mathbb{S}^n\to\mathbb{R},\qquad
\varphi(p)=\langle X(p),E_1\rangle=x_1.
\]
Throughout the example, all gradients, Hessians, and Laplacians are taken with respect to the standard round metric \(g\) on \(\mathbb{S}^n\).

The tangential projection of \(E_1\) onto \(T_p\mathbb{S}^n\) gives
\[
\nabla\varphi=E_1-\varphi X.
\]
Consequently,
\begin{equation}\label{eq:nabla phi}
|\nabla\varphi|^2=1-\varphi^2.
\end{equation}
Moreover, the standard identities for the coordinate functions on \(\mathbb{S}^n\) yield
\[
\nabla^2\varphi=-\varphi g,
\qquad
\Delta\varphi=-n\varphi.
\]
See, for instance, \cite[Lemma 8.1-8.3]{Peter Li}.

Take
\[
v_\varepsilon=1+\varepsilon\varphi,
\qquad
u_\varepsilon=v_\varepsilon^2,
\qquad |\varepsilon|<\frac12.
\]
Since \(-1\le \varphi\le 1\), we have
\[
v_\varepsilon\ge 1-|\varepsilon|>0,
\]
and hence
\[
\ln u_\varepsilon=2\ln v_\varepsilon.
\]
First, we compute the gradient term,
\begin{equation*}
	\int_{\mathbb{S}^n}|\nabla\ln u_\varepsilon|^2u_\varepsilon\,d\sigma
	=4\int_{\mathbb{S}^n}|\nabla v_\varepsilon|^2\,d\sigma
	=4\varepsilon^2\int_{\mathbb{S}^n}|\nabla\varphi|^2\,d\sigma
	=-4\varepsilon^2\int_{\mathbb{S}^n}\varphi\,\Delta\varphi\,d\sigma
	=4n\varepsilon^2\int_{\mathbb{S}^n}\varphi^2\,d\sigma.
\end{equation*}

Next, we consider the Hessian term,
\[
\nabla^2\ln u_\varepsilon
=2\nabla^2\ln v_\varepsilon
=2\left(
\frac{\nabla^2 v_\varepsilon}{v_\varepsilon}
-\frac{dv_\varepsilon\otimes dv_\varepsilon}{v_\varepsilon^2}
\right).
\]
Since
\[
v_\varepsilon=1+\varepsilon\varphi,
\qquad
dv_\varepsilon=\varepsilon\,d\varphi,
\qquad
\nabla^2 v_\varepsilon=\varepsilon\nabla^2\varphi=-\varepsilon\varphi g,
\]
we get
\begin{equation*}
	\nabla^2\ln u_\varepsilon
	=-\frac{2\varepsilon\varphi}{v_\varepsilon}g
	-\frac{2\varepsilon^2}{v_\varepsilon^2}d\varphi\otimes d\varphi.
\end{equation*}
Thus
\begin{align*}
	|\nabla^2\ln u_\varepsilon|^2
	=4\left(
	\frac{n\varepsilon^2\varphi^2}{v_\varepsilon^2}
	+\frac{2\varepsilon^3\varphi}{v_\varepsilon^3}|\nabla\varphi|^2
	+\frac{\varepsilon^4}{v_\varepsilon^4}|\nabla\varphi|^4
	\right),
\end{align*}
and
\begin{align*}
	u_\varepsilon|\nabla^2\ln u_\varepsilon|^2
	=4\left(
	n\varepsilon^2\varphi^2
	+\frac{2\varepsilon^3\varphi}{v_\varepsilon}|\nabla\varphi|^2
	+\frac{\varepsilon^4}{v_\varepsilon^2}|\nabla\varphi|^4
	\right).
\end{align*}
Using \eqref{eq:nabla phi}, we obtain
\begin{equation*}
	\frac14\int_{\mathbb{S}^n}|\nabla^2\ln u_\varepsilon|^2u_\varepsilon\,d\sigma
	=n\varepsilon^2\int_{\mathbb{S}^n}\varphi^2\,d\sigma+2\varepsilon^3\int_{\mathbb{S}^n}
	\frac{\varphi(1-\varphi^2)}{1+\varepsilon\varphi}\,d\sigma+\varepsilon^4\int_{\mathbb{S}^n}
	\frac{(1-\varphi^2)^2}{(1+\varepsilon\varphi)^2}\,d\sigma.
\end{equation*}

For \(0<\varepsilon\le \tfrac12\), we have \(|\varphi|\le1\) and
\[
v_\varepsilon=1+\varepsilon\varphi\ge1-\varepsilon\ge\frac12,
\qquad 0\le 1-\varphi^2\le1.
\]
Therefore
\[
\left|\frac{\varphi(1-\varphi^2)}{1+\varepsilon\varphi}\right|
\le \frac{|\varphi|(1-\varphi^2)}{1+\varepsilon\varphi}
\le 2|\varphi|(1-\varphi^2)\le2,
\]
and
\[
0\le \frac{(1-\varphi^2)^2}{(1+\varepsilon\varphi)^2}\le4.
\]
It follows that
\begin{equation*}
	\int_{\mathbb{S}^n}|\nabla^2\ln u_\varepsilon|^2u_\varepsilon\,d\sigma
	=4n\varepsilon^2\int_{\mathbb{S}^n}\varphi^2\,d\sigma
	+O(\varepsilon^3),
\end{equation*}
whereas
\begin{equation*}
	\int_{\mathbb{S}^n}|\nabla\ln u_\varepsilon|^2u_\varepsilon\,d\sigma
	=4n\varepsilon^2\int_{\mathbb{S}^n}\varphi^2\,d\sigma
	=O(\varepsilon^2).
\end{equation*}
Hence, for \(\varepsilon\ne0\),
\[
\frac{
	\displaystyle
	\int_{\mathbb{S}^n}
	|\nabla^2\ln u_\varepsilon|^2
	u_\varepsilon\,d\sigma
}{
	\displaystyle
	\int_{\mathbb{S}^n}
	|\nabla\ln u_\varepsilon|^2
	u_\varepsilon\,d\sigma
}
=
1+O(\varepsilon).
\]
Sharpness is shown by
\begin{equation}\label{eq:optimal-limit}
	\lim_{\varepsilon\to0}
	\frac{\displaystyle\int_{\mathbb{S}^n}|\nabla^2\ln u_\varepsilon|^2u_\varepsilon\,d\sigma}
	{\displaystyle\int_{\mathbb{S}^n}|\nabla\ln u_\varepsilon|^2u_\varepsilon\,d\sigma}
	=\frac{\lambda_1(\mathbb{S}^n)}{n}=1.
\end{equation}
That is, if there existed \(\delta>0\) such that
\[
\int_{\mathbb{S}^n}|\nabla^2\ln u|^2u\,d\sigma
\ge (1+\delta)\int_{\mathbb{S}^n}|\nabla\ln u|^2u\,d\sigma
\]
for every positive function \(u\).
Then applying this inequality to the above \(u_\varepsilon\) and taking \(\varepsilon\to0\) would contradict \eqref{eq:optimal-limit}.
\end{remark}
\begin{remark}
	In particular, the lemma can apply to any smooth positive solution \(u=u(x,t)\) of the heat equation
	\[
	\partial_tu=\Delta_gu
	\]
	for all \(t>0\). The constant depends only on the Riemannian manifold and is independent of both the solution \(u\) and the time \(t\).
\end{remark}

Based on Lemma \ref{Lem: poinca th1}, we can now prove Theorem \ref{th: concave1}.

\begin{proof}[Proof of Theorem \ref{th: concave1}]
	We divide the proof into two cases, corresponding to the assumptions (a) and (b).\\
	{\bf{Case (a):}}
	By the definition \eqref{def:Ric}, we have
	$$
	\operatorname{Ric}(\nabla \ln u, \nabla \ln u) \geq - \|\operatorname{Ric}^-\|_{L^\infty} |\nabla \ln u|^2.
	$$
	In the case $\partial \mathbf{M} = \emptyset$, substituting this lower bound into the second derivative formula of $E(t)$ (i.e., \eqref{eq:E''t}) yields
	$$
	E''(t) \leq -2 \int_{\mathbf{M}} |\nabla^2 \ln u|^2 u \, d\mathrm{vol}_g + 2 \int_{\mathbf{M}} \|\operatorname{Ric}^-\|_{L^\infty} |\nabla \ln u|^2 u \, d\mathrm{vol}_g.
	$$
	Applying our hypothesis $\|\operatorname{Ric}^{-}\|_{L^{\infty}} \leq \frac{\lambda_1(\mathbf{M}, g)}{n}$, we have
	\begin{equation}\label{E bound}
		E''(t) \leq -2 \int_{\mathbf{M}} |\nabla^2 \ln u|^2 u \, d\mathrm{vol}_g + \frac{2\,\lambda_1(\mathbf{M}, g)}{n} \int_{\mathbf{M}} |\nabla \ln u|^2 u \, d\mathrm{vol}_g.
	\end{equation}
	
	We now apply the result of Lemma \ref{Lem: poinca th1}. For all $t> 0$, the gradient term is controlled by the Hessian term
	$$
	\int_{\mathbf{M}} |\nabla \ln u|^2 u \, d\mathrm{vol}_g \leq \frac{n}{\lambda_1(\mathbf{M}, g)} \int_{\mathbf{M}} |\nabla^2 \ln u|^2 u \, d\mathrm{vol}_g.
	$$
	Substituting this inequality into \eqref{E bound}, we obtain that
	$$
	E''(t) \leq -2 \int_{\mathbf{M}} |\nabla^2 \ln u|^2 u \, d\mathrm{vol}_g + \frac{2\,\lambda_1(\mathbf{M}, g)}{n}\left( \frac{n}{\lambda_1(\mathbf{M},g)} \int_{\mathbf{M}} |\nabla^2 \ln u|^2 u \, d\mathrm{vol}_g \right).
	$$
	Hence
	$$E''(t) \leq 0.$$
	{\bf{Case (b):}} The proof for case (b) proceeds with the same notation as in case (a). We assume that the second fundamental form satisfies
	$
	\mathrm{II} \ge -\sigma g_{\partial M}, \, \sigma \ge 0.
	$
	It follows that
	$$
	-2\mathrm{II}\left(\nabla^\partial f, \nabla^\partial f\right) \le 2\sigma\vert{}\nabla^\partial f\vert{}^2.
	$$
	Substituting this back into the $E''(t)$ (i.e., \eqref{eq:E''t}) yields
	\begin{equation}\label{bdEbd}
		\frac{1}{2}E''(t) \le -\int_\mathbf{M} u\left[ \vert{}\nabla^2\ln u\vert{}^2 + \operatorname{Ric}(\nabla\ln u, \nabla\ln u) \right]\,d\mathrm{vol}_g+ \sigma\int_{\partial \mathbf{M}} u\vert{}\nabla^\partial\ln u\vert{}^2 \,dA_g.
	\end{equation}
	To absorb the boundary integral term, we apply the trace inequality \eqref{trace ineq} to $\left| \nabla v\right| $, where $v=\sqrt{u}$ and $u\in C^2(\mathbf{M})$. This gives
	\begin{equation}\label{boundary ineq.}
		\frac{1}{4}\int_{\partial \mathbf{M}} \dfrac{|\nabla u|^2}{u}\,dA_g \;\le\; \frac1{\lambda_{tr}}\int_\mathbf{M}\Bigl(\frac{1}{4}\,\dfrac{|\nabla u|^2}{u} + |\nabla^2 v|^2\Bigr)\,d\mathrm{vol}_g.
	\end{equation}

	Because of the term $\int_\mathbf{M}|\nabla^2 v|^2\,d\mathrm{vol}_g$ appearing in the above inequality, the proof requires a slight modification of the argument used in case (a). We first derive a relation between \(\int_\mathbf{M}|\nabla^2 v|^2\,d\mathrm{vol}_g\) and \(\int_{\mathbf{M}} |\nabla^2 \ln u|^2 u \, d\mathrm{vol}_g\), and then apply case (b) of Lemma \ref{Lem: poinca th1} to conclude the proof.
	Next, we establish an upper bound for the Hessian of $v$. Writing $v = e^\frac{\ln u}{2}$, a direct calculation yields
	\[
	\nabla^2 v = v\left(\frac12\nabla^2 \ln u + \frac14 d\ln u\otimes d\ln u\right) .
	\]
	Using the elementary inequality $\left| S+T\right| ^2 \le 2\left| S\right| ^2 + 2\left| T\right| ^2$ for the symmetric $2$-tensors $$
	S=\frac12\nabla^2\ln u,\qquad T=\frac14 d\ln u\otimes d\ln u,$$
	together with the identity $\left| d\ln u \otimes d\ln u\right| ^2 = \left| \nabla \ln u\right| ^4$, we find that
	$$
	\left| \nabla^2 v\right| ^2 \le \frac12 u \left| \nabla^2 \ln u\right| ^2 + \frac18 u \left| \nabla \ln u\right| ^4.
	$$
	Integrating this pointwise estimate over $\mathbf{M}$ yields
	\begin{equation}\label{hessian vf}
		\int_\mathbf{M} \left| \nabla^2 v\right| ^2 \, d\mathrm{vol}_g \le \frac{1}{2} \int_\mathbf{M}u \left| \nabla^2 \ln u\right| ^2\, d\mathrm{vol}_g + \frac18 \underbrace{	\int_\mathbf{M}u \left| \nabla \ln u\right| ^4\, d\mathrm{vol}_g}_{Q},
	\end{equation}
	where we define the integral quantity
	$$
	Q:= \int_\mathbf{M} u \left| \nabla \ln u\right| ^4 \, d\mathrm{vol}_g.
	$$
	To control $Q$, since $u$ satisfies the Neumann boundary condition $\partial_\nu u = 0$ on $\partial \mathbf{M}$, the divergence theorem implies
	$$
	\begin{aligned}
		Q
		=\int_\mathbf{M}
		\dfrac{|\nabla u|^4}{u^3}\,d\mathrm{vol}_g
		&=\int_\mathbf{M} \langle \nabla u,
		\frac{\nabla u}{u}\rangle  \frac{|\nabla u|^2}{u^2}\,d\mathrm{vol}_g\\
		&=-\int_\mathbf{M} u\,(\Delta \ln u)\,\frac{|\nabla u|^2}{u^2}\,d\mathrm{vol}_g- \int_\mathbf{M}  u\left\langle \nabla \ln u , \nabla(|\nabla \ln u|^2)\right\rangle\,d\mathrm{vol}_g.
	\end{aligned}
	$$
	Applying the Cauchy-Schwarz inequality to the integral on the right-hand side yields
	$$
	\begin{aligned}
		Q \le &\left( \int_\mathbf{M} u\,(\Delta \ln u)^2\,d\mathrm{vol}_g\right)^\frac{1}{2}\left(\int_\mathbf{M}
		\dfrac{|\nabla u|^4}{u^3}\,d\mathrm{vol}_g \right)^\frac{1}{2}\\
		&+2\left( \int_\mathbf{M} (\nabla|\nabla \ln u|)^2 u\,d\mathrm{vol}_g\right)^\frac{1}{2}\left( \int_\mathbf{M}
		|\nabla \ln u|^4 u\,d\mathrm{vol}_g\right) ^\frac{1}{2}.
	\end{aligned}
	$$
	By using the pointwise bounds $(\Delta \ln u)^2 \le  n\,\left|  \nabla^2 \ln u\right| ^2$ and $|\nabla\left| \nabla \ln u|\right|^2\le | \nabla^2\ln u|^2$, it follows that
	\begin{equation}
		Q \le \sqrt{n}\left( \int_\mathbf{M} |\nabla^2 \ln u|^2 u\,d\mathrm{vol}_g\right)^\frac{1}{2}
		Q^\frac{1}{2}+2 \left( \int_\mathbf{M} |\nabla^2 \ln u|^2 u\,d\mathrm{vol}_g\right)^\frac{1}{2}Q^\frac{1}{2}.
	\end{equation}
	Since $Q > 0$ (the case $Q= 0$ being immediate), dividing by $Q^\frac{1}{2}$ and squaring both sides leads to
	\begin{equation}\label{B bound}
		Q= \int_\mathbf{M} u \left| \nabla \ln u\right| ^4 \, d\mathrm{vol}_g \le (\sqrt{n}+2)^2 \int_\mathbf{M} u \left| \nabla^2 \ln u\right| ^2 \, d\mathrm{vol}_g.	
	\end{equation}
	Substituting estimate \eqref{B bound} back into \eqref{hessian vf}, we obtain
	\begin{equation}\label{hessv}
		\int_\mathbf{M} \left| \nabla^2 v\right| ^2 \, d\mathrm{vol}_g \le \frac{n+4\sqrt n+8}{8} \int_\mathbf{M} u \left| \nabla^2 \ln u\right| ^2 \, d\mathrm{vol}_g.
	\end{equation}
The above inequalities \eqref{B bound} and \eqref{hessv} already appeared in the papers \cite{CGK} and \cite{CFHS} Appendix A.
	Combining \eqref{boundary ineq.} with \eqref{hessv}, we obtain that
	$$
	\sigma\int_{\partial \mathbf{M}} u\vert{}\nabla^\partial\ln u\vert{}^2 \,dA_g\leq\frac{(n+4\sqrt n+8)\sigma}{2\lambda_{tr}}\int_\mathbf{M} u \vert{}\nabla^2\ln u\vert{}^2 \,d\mathrm{vol}_g+ 	\frac{\sigma}{\lambda_{tr}}\int_{\mathbf{M}} \dfrac{|\nabla u|^2}{u}\,d\mathrm{vol}_g.
	$$
	Moreover, by the definition of $\operatorname{Ric}^{-}$, we have
	$$
	\operatorname{Ric}(\nabla \ln u, \nabla \ln u) \geq - \|\operatorname{Ric}^-\|_{L^\infty} |\nabla \ln u|^2.
	$$
	Substituting this curvature lower bound and the boundary term estimate into \eqref{bdEbd}  leads to
	\begin{equation*}
		\frac{1}{2}E''(t)
		\leq\left( \frac{(n+4\sqrt n+8)\sigma}{2\lambda_{tr}} -1\right)  \int_{\mathbf{M}} |\nabla^2 \ln u|^2 u \, d\mathrm{vol}_g + ( \frac{\sigma}{\lambda_{tr}}+\|\operatorname{Ric}^-\|_{L^\infty})  \int_{\mathbf{M}} |\nabla \ln u|^2 u \, d\mathrm{vol}_g.
	\end{equation*}
	
	Next, we recall result \eqref{ineq:neu poin} from Lemma~\ref{Lem: poinca th1}, i.e.,
	$$
	\int_\mathbf{M}|\nabla\ln u|^2 \, u \, d\mathrm{vol}_g
	\le
	\frac{n}
	{\lambda_1^N(\mathbf{M}, g)}
	\int_\mathbf{M}|\nabla^2\ln u|^2 \, u \, d\mathrm{vol}_g.
	$$
	Combining this with the previous inequality yields
	\begin{equation}\label{con bd ineq }
		\frac{1}{2}E''(t) \leq\small{\left( \frac{(n+4\sqrt n+8)\sigma}{2\lambda_{tr}} -1+(\frac{\sigma}{\lambda_{tr}}+\|\operatorname{Ric}^-\|_{L^\infty})\frac{n}
			{\lambda_1^N(\mathbf{M},g)}\right)} \int_{\mathbf{M}} |\nabla^2 \ln u|^2 u \, d\mathrm{vol}_g.
	\end{equation}
	
	Under the assumption that
	$$
	\|\operatorname{Ric}^-\|_{L^\infty} \leq \frac{\lambda_1^N(\mathbf{M}, g)}{n}
	-\frac{\sigma\lambda_1^N(\mathbf{M}, g)(n+4\sqrt n+8)}{2n\lambda_{tr}}-\frac{\sigma}{\lambda_{tr}},
	$$
	the coefficient on the right-hand side of \eqref{con bd ineq } is nonpositive, ensuring that
	$$
	E''(t) \leq 0.
	$$
	The proof is complete.
\end{proof}
\subsection{Noncompact Case}
We can't prove that the second order logarithmic Poincar\'e inequality holds on noncompact manifolds for all positive solutions of the heat equation.  However, as given in next lemma, it holds for the heat kernel $u= G(x, t, y)$ on  some noncompact manifolds in short time.
\begin{lemma}\label{Wpoinca th}
	Let $(\mathbf{M}^n, g)$ be a complete noncompact Riemannian manifold satisfying
	\begin{itemize}
		\item[(a)] $\left|\nabla^j \operatorname{Rm}\right| \leq K,\quad j=0,1,2,3,$
		\item[(b)] the volume noncollapsing condition
		$\inf\limits_{y\in \mathbf{M}} |B(y, 1)| \geq \kappa> 0.$
	\end{itemize}
	Let $u=G(x, t, y)$ be the heat kernel with pole at $y$. Then for any given constant $\delta > 0$, there exists a constant
	$$
	t_*=t_*\left(n, K, \kappa, \delta\right)\in(0,1]
	$$
	such that
	\begin{equation}\label{Wpoinca ineq}
		\int_\mathbf{M}|\nabla \ln u|^2 u\,d\mathrm{vol}_g \leq (2+ \delta) t \int_\mathbf{M}\left|\nabla^2 \ln u\right|^2 u \,d\mathrm{vol}_g,\qquad \forall t \in(0, t_*].
	\end{equation}
\end{lemma}
\begin{proof}[Proof of Lemma \ref{Wpoinca th}]
	By the method of contradiction, if not,  then there exists a fixed $\delta > 0$ such that for any arbitrarily small time $t_*$, the inequality must fail for some manifold satisfying the hypotheses at some time $t \in (0, t_*]$. In particular, we can find sequences of times $t_i \to 0$, and there exists a corresponding sequence of pointed complete Riemannian manifolds $(\mathbf{M}_i, g_i, y_i)$ such that for $u_i = G_i(\cdot, t_i, y_i)$, we have
	\begin{equation}\label{eq:contradiction-original}
	\int_{\mathbf{M}_i}\left|\nabla \ln u_i\right|^2 u_i\,d\mathrm{vol}_{g_i} > (2+\delta) t_i \int_{\mathbf{M}_i}\left|\nabla^2 \ln u_i\right|^2 u_i \,d\mathrm{vol}_{g_i}.
	\end{equation}
	
	Consider the scaled metrics $\tilde{g}_i=\frac{1}{t_i} g_i,  \tilde{t}=\frac{t}{t_i}$. Under this scaling, the geometric quantities transform as follows
	$$
	\left|\tilde{\nabla}^j \operatorname{Rm}(\tilde{g}_i)\right|_{\tilde{g}_i} = t_i^{1+j/2}\left|\nabla^j \operatorname{Rm}(g_i)\right|_{g_i} \leq t_i^{1+j/2} K \longrightarrow 0\quad \text{as}\,\, i \to \infty.
	$$
	By standard injectivity radius estimate \cite[Theorem 5.42]{Hamilton Ricci flow}, the uniform curvature bounds and the uniform volume noncollapsing condition $\inf |B(y, 1)| \ge \kappa$ guarantee a uniform lower bound on the injectivity radius 
	$$
\operatorname{inj}_{g_i}(y_i) \ge \iota_0 > 0,
	$$
	where $\iota_0$ depends on $n,K,\kappa$.
	
	Under the scaled metrics,
	$$
\operatorname{inj}_{\tilde{g}_i}(y_i) = \frac{1}{\sqrt{t_i}} \operatorname{inj}_{g_i}(y_i) \ge \frac{\iota_0}{\sqrt{t_i}} \longrightarrow \infty.
	$$
	By the Cheeger-Gromov Compactness Theorem, there exists a subsequence (still denoted by $i$) that converges in the 
	$C_{loc}^{2, \alpha}$ topology to the standard flat Euclidean space
	$$
	\left(\mathbf{M}_i, \tilde{g}_i, y_i\right) \xrightarrow{C_{loc}^{2, \alpha}} \left(\mathbf{M}_{\infty}, g_{\infty}, y_{\infty}\right) = \left(\mathbb{R}^n, g_\mathbb{\small E}, 0\right).
	$$
	
	We define the scaled heat kernel $\tilde{u}_i(x, \tilde{t}) = t_i^{n/2} u_i(x, t_i \tilde{t})$, which is the fundamental solution to the heat equation on $(\mathbf{M}_i, \tilde{g}_i)$. Note that $\ln \tilde{u}_i = \ln u_i + \frac{n}{2}\ln t_i$, so the spatial derivatives are identical. Under the scaled metrics $\tilde{g}_i$,
	
	$$
	|\tilde{\nabla} \ln \tilde{u}_i|^2_{\tilde{g}_i} = t_i |\nabla \ln u_i|^2_{g_i}, \quad |\tilde{\nabla}^2 \ln \tilde{u}_i|^2_{\tilde{g}_i} = t_i^2 |\nabla^2 \ln u_i|^2_{g_i}, \quad d\mathrm{vol}_{\tilde{g}_i} = t_i^{-n/2} d\mathrm{vol}_{g_i}.
	$$
	Integrating these over the manifolds yields
	$$
	\int_{\mathbf{M}_i} |\tilde{\nabla} \ln \tilde{u}_i|^2_{\tilde{g}_i} \tilde{u}_i \,d\mathrm{vol}_{\tilde{g}_i} = t_i \int_{\mathbf{M}_i} |\nabla \ln u_i|^2_{g_i} u_i \,d\mathrm{vol}_{g_i},
	$$
	
	$$
	\int_{\mathbf{M}_i} |\tilde{\nabla}^2 \ln \tilde{u}_i|^2_{\tilde{g}_i} \tilde{u}_i \,d\mathrm{vol}_{\tilde{g}_i} = t_i^2 \int_{\mathbf{M}_i} |\nabla^2 \ln u_i|^2_{g_i} u_i \,d\mathrm{vol}_{g_i}.
	$$
	By the scaling identities established above, the \eqref{eq:contradiction-original} becomes, at $\tilde{t}=1$,
	\begin{equation}\label{scaled ineq}
		\int_{\mathbf{M}_i} \left|\tilde{\nabla} \ln \tilde{u}_i\right|_{\tilde{g}_i}^2 \tilde{u}_i \,d\mathrm{vol}_{\tilde{g}_i} > (2+\delta) \int_{\mathbf{M}_i} \left|\tilde{\nabla}^2 \ln \tilde{u}_i\right|_{\tilde{g}_i}^2 \tilde{u}_i \,d\mathrm{vol}_{\tilde{g}_i}.
	\end{equation}
	
	Because both the domains and the measures vary with \(i\), the dominated convergence theorem cannot be applied directly to these integrals when passing the limit $i \to \infty$ inside the integrals of \eqref{scaled ineq}. We instead prove convergence on fixed compact subsets after pullback by the Cheeger-Gromov convergence maps and then establish a tail (the integrals outside of arbitrarily large balls) estimate that is uniform in \(i\).
	
	Hamilton's gradient estimate \cite[Theorem~1.1]{Hamilton} in the complete noncompact setting (see \cite{Kots}) gives, for \(t\in (\frac{t_i}{2}, t_i]\),
	$$
	\left|\nabla \ln u_i\right|^2 u_i=\frac{\left|\nabla u_i\right|^2}{u_i} \leq \tilde{c} \frac{u_i}{t-\frac{t_i}{2}} \ln \frac{A_i}{u_i},
	$$
	where $\tilde{c}= \tilde{c}(n, K, \kappa)$ and
	\begin{equation}\label{def: A_i}
		A_i = \sup\limits_{\mathbf{M}_i\times[\frac{t_i}{2}, t_i]} u_i(\cdot, t).
	\end{equation}

	We next derive a two-sided bound for the eigenvalues of
	\(\nabla^2\ln u_i\). For \(t\in (\frac{t_i}{2}, t_i]\), the global Hessian upper estimate of Han-Zhang \cite{Zhang and Han}
	implies
	\[
	\nabla^2\ln u_i
	\leq
	\frac{C}{t-\frac{t_i}{2}}
	\left(1+\ln\frac{A_i}{u_i}\right)g_i,
	\]
	where $C=C(n, K).$
	By the uniform version of Hamilton's matrix lower estimate \cite[Theorem 4.3]{Hamilton}, we have
	
	\[
	\nabla^2\ln u_i
	\geq
	-\frac{1}{2(t-\frac{t_i}{2})}g_i
	-
	C\left(
	1+\ln\frac{A_i}{u_i}
	\right)g_i,
	\]
	where $C$ depends only on $n$, $K$, and $\kappa$. Since the limiting manifold has bounded geometry, these matrix  estimates, which were originally proven for compact manifolds, still hold currently.
	
	We next establish uniform pointwise estimates which will be used to
	control the tails of the two integrals in \eqref{scaled ineq}.
	By Bishop-Gromov volume comparison theorem (e.g.  \cite[Theorem~1.1]{WeiVolumeComp}) and the noncollapsing assumption, we have
	\begin{equation}\label{eq:vol}
		c(n,K,\kappa)r^n\le|B_{g_i}(x,r)|\le C(n,K)r^n, \qquad
		x\in\mathbf M_i, \quad 0<r\le1.
	\end{equation}
Here and below, $c, C>0$ denote uniform constants, possibly varying
from line to line, which depend only on $n, K, \kappa$ and are
independent of $i$.

Combining \eqref{eq:vol} with the standard short-time two-sided
Gaussian heat kernel estimates
(see, e.g. \cite{Peter Li and Yau}),
we obtain
\begin{equation}\label{eq:Gaussian-original}
	ct^{-n/2}
	\exp\left(-C\frac{d_{g_i}(x,y_i)^2}{t}\right)
	\le
	u_i(x,t)
	\le
	Ct^{-n/2}
	\exp\left(-c\frac{d_{g_i}(x,y_i)^2}{t}\right),
	\quad
	0<t\le1.
\end{equation}
	The two-sided short-time heat kernel bounds imply
	\[
	\frac{c}{t_i^{n/2}}\leq A_i\leq \frac{C}{t_i^{n/2}},
	\]
	where $c$ and $C$ also depend on $n,K,\kappa$.
 Hence
\[
|\nabla^2\ln u_i|^2u_i
\leq
\frac{\Upsilon}{(t-\frac{t_i}{2})^2}
u_i\left(1+\ln\frac{A_i}{u_i}\right)^2,
\]
where $\Upsilon= \Upsilon(n, K, \kappa).$
When $t\in[\frac{3t_i}{4}, t_i]$, we have $\frac{t_i}{4} \leq t-\frac{t_i}{2}\leq \frac{t_i}{2}$, and hence the following pointwise estimates under the unscaled metrics
	$$
	\begin{aligned}
		& \left|\nabla \ln u_i\right|^2 u_i=\frac{\left|\nabla u_i\right|^2}{u_i} \leq \tilde{c} \frac{u_i}{t_i} \ln \frac{A_i}{u_i},
		\\
		&\left|\nabla^2 \ln u_i\right|^2 u_i \leq
		\dfrac{\Upsilon}{t_i^2} u_i (1+ \ln \frac{A_i}{u_i})^2,
	\end{aligned}
	$$
	where $A_i$ is given in \eqref{def: A_i}. Transforming these bounds into the scaled geometry at $\tilde{t} = 1$, we get
	$$
	\left|\tilde{\nabla} \ln \tilde{u}_i\right|_{\tilde{g}_i}^2 \tilde{u}_i \leq \tilde{c}  \tilde{u}_i \ln\frac{\tilde{A_i}}{\tilde{u}_i},
	$$
	$$
	\left|\tilde{\nabla}^2 \ln \tilde{u}_i\right|_{\tilde{g}_i}^2 \tilde{u}_i \leq \Upsilon \tilde{u}_i (1 + \ln\frac{\tilde{A_i}}{\tilde{u}_i})^2,
	$$
	where $\tilde{A_i} = t_i^{n/2} A_i \leq C(n,K,\kappa)$.
	
	Moreover, under the scaled metrics, the volume comparison theorem \eqref{eq:vol} implies that the unit balls have uniform volume bounds
	\[
	c
	\leq
	|B_{\widetilde g_i}(x,1)|
	\leq
	C.
	\]
	From the scaled Li-Yau estimate, the uniform two-sided Gaussian heat kernel estimates therefore give
	\begin{equation}\label{eq:scaled-gaussian}
		c\,e^{-Cr_i(x)^2}
		\leq
		\widetilde u_i(x,1)
		\leq
		C\,e^{-cr_i(x)^2},
		\qquad r_i(x)
		= d_{\widetilde g_i}(x,y_i).
	\end{equation}
	The lower bound implies
	$\ln(\tilde A_i/\tilde u_i)\le C(1+r_i^2)$.
	Combining with the upper
	bound, we obtain
	\begin{align*}
		|\tilde\nabla\ln\tilde u_i|^2\tilde u_i
		&\le C(1+r_i^2)e^{-cr_i^2},\\
		|\tilde\nabla^2\ln\tilde u_i|^2\tilde u_i
		&\le C(1+r_i^2)^2e^{-cr_i^2}.
	\end{align*}

	Define
	\[
	F_i
	= |\widetilde\nabla\ln\widetilde u_i(\cdot,1)|_{\widetilde g_i}^2
	\widetilde u_i(\cdot,1),
	\]
	and
	\[
	H_i
	= |\widetilde\nabla^2\ln\widetilde u_i(\cdot,1)|_{\widetilde g_i}^2
	\widetilde u_i(\cdot,1).
	\]
	The preceding estimates give
	\begin{equation}\label{eq:pointwise-domination}
		0\leq F_i(x)+H_i(x)
		\leq
		C\bigl(1+r_i(x)^4\bigr)e^{-cr_i(x)^2},
	\end{equation}

	We first prove convergence on compact subsets. Fix $R>0$.
	By pointed Cheeger-Gromov convergence theorem, for all sufficiently large $i$, there exist smooth embeddings
	\[
	\Phi_{i,R}
	:
	B_{\mathbb R^n}(0,3R)\longrightarrow \mathbf{M}_i,
	\qquad
	\Phi_{i,R}(0)=y_i,
	\]
	such that
	\[
	h_i:=\Phi_{i,R}^{*}\widetilde g_i
	\longrightarrow g_{\mathbb E}
	\]
	in $C^{2,\alpha}$ on $B_{\mathbb R^n}(0,3R)$. After enlarging the
	domain, we may assume that
	\[
	B_{\widetilde g_i}(y_i,2R)
	\subset
	\Phi_{i,R}\bigl(B_{\mathbb R^n}(0,3R)\bigr)
	\]
	for all sufficiently large $i$.
	Let
	\[
	q_i(x)
	= \widetilde u_i(\Phi_{i,R}(x),1).
	\]
	Then
	\[
	q_i\longrightarrow u_\infty(\cdot,1)
	\]
	in $C^2(B_{\mathbb R^n}(0,3R))$. Since
	$u_\infty(\cdot,1)\geq c>0$  on this fixed ball,
	we also have
	\[
	\ln q_i
	\longrightarrow
	\ln u_\infty(\cdot,1)
	\]
	in $C^2(B_{\mathbb R^n}(0,3R))$. Consequently,
	\[
	F_i\circ\Phi_{i,R}
	\longrightarrow
	F_\infty
	:=
	|\nabla\ln u_\infty(\cdot,1)|^2u_\infty(\cdot,1),
	\]
	and
	\[
	H_i\circ\Phi_{i,R}
	\longrightarrow
	H_\infty
	:=
	|\nabla^2\ln u_\infty(\cdot,1)|^2u_\infty(\cdot,1)
	\]
	uniformly on $B_{\mathbb R^n}(0,3R)$. Moreover,
	\[
	d\mathrm{vol}_{h_i}\longrightarrow d\mathrm{vol}_{g_{\mathbb E}}
	\]
	uniformly on the same ball.
	
	Choose a continuous function
	$\eta:[0,\infty)\to[0,1]$ satisfying
	\[
	\eta\equiv1\quad\text{on }[0,1],
	\qquad
	\eta\equiv0\quad\text{on }[2,\infty),
	\]
	and define
	\[
	\eta_{i,R}(x)
	= \eta\left(\frac{r_i(x)}{R}\right),
	\qquad
	\eta_{\infty,R}(x)
	= \eta\left(\frac{|x|}{R}\right).
	\]
	The convergence of the metrics implies that
	\[
	\eta_{i,R}\circ\Phi_{i,R}
	\longrightarrow
	\eta_{\infty,R}
	\]
	uniformly on the fixed coordinate ball. Hence
	\begin{equation}\label{eq:truncated-F}
		\int_{\mathbf{M}_i}\eta_{i,R}F_i\,d\mathrm{vol}_{\tilde{g}_i}=\int_{B_{\mathbb R^n}(0,3R)}
		(\eta_{i,R}F_i)\circ\Phi_{i,R}\,d\mathrm{vol}_{h_i}
		\longrightarrow
		\int_{\mathbb R^n}\eta_{\infty,R}F_\infty\,d\mathrm{vol}_{g_{\mathbb E}},
	\end{equation}
	and
	\begin{equation}\label{eq:truncated-H}
		\int_{\mathbf{M}_i}\eta_{i,R}H_i\,d\mathrm{vol}_{\tilde{g}_i}=\int_{B_{\mathbb R^n}(0,3R)}
		(\eta_{i,R}H_i)\circ\Phi_{i,R}\,d\mathrm{vol}_{h_i}
		\longrightarrow
		\int_{\mathbb R^n}\eta_{\infty,R}H_\infty\,d\mathrm{vol}_{g_{\mathbb E}}.
	\end{equation}

	It remains to show that the tails vanish uniformly in $i$. Since
	\[
	\operatorname{Ric}(g_i)\geq-\Lambda g_i
	\]
	and \(t_i\leq1\), then we have
	\[
	\operatorname{Ric}(\widetilde g_i)
	=
	\operatorname{Ric}(g_i)
	\geq
	-\Lambda t_i\widetilde g_i
	\geq
	-\Lambda \widetilde g_i,
	\]
	for some constant $\Lambda$ independent of $i$, Bishop-Gromov
	comparison theorem gives
	\begin{equation}\label{eq:scaled-volume-growth}
		\operatorname{Vol}_{\widetilde g_i}
		B_{\widetilde g_i}(y_i,\rho)
		\leq
		\tilde{C}(1+\rho)^ne^{\tilde{C}\rho},
		\qquad
		\rho\geq1.
	\end{equation}

	For each integer $m\geq 1$, let
	\[
	\mathcal A_{i,m}
	= B_{\widetilde g_i}(y_i,m+1)
	\setminus
	B_{\widetilde g_i}(y_i,m).
	\]
	On $\mathcal A_{i,m}$,
	\eqref{eq:pointwise-domination} and
	\eqref{eq:scaled-volume-growth} give
	\[
	\begin{aligned}
		\int_{\mathcal A_{i,m}}(F_i+H_i)\,
		d\mathrm{vol}_{\tilde{g}_i}
		&\leq
		C(1+m^4)e^{-cm^2}
		\operatorname{Vol}_{\widetilde g_i}
		B_{\widetilde g_i}(y_i,m+1)\\
		&\leq
		C(1+m^{n+4})e^{-cm^2+\tilde{C}m}.
	\end{aligned}
	\]
	For all sufficiently large $m$, we have
	\[
	-cm^2+\tilde{C}m
	\leq
	-\frac{c}{2}m^2.
	\]
	We obtain
	\[
	\int_{\mathcal A_{i,m}}(F_i+H_i)\,
	d\mathrm{vol}_{\tilde{g}_i}
	\leq
	C(1+m^{n+4})e^{-cm^2/2}.
	\]
	Summing these estimates over all
	$m$ gives
	\[
	\begin{aligned}
		\int_{\mathbf{M}_i\setminus B_{\widetilde g_i}(y_i,R)}
		(F_i+H_i)\,d\mathrm{vol}_{\tilde{g}_i}
		&\leq
		C\sum_{m=\lfloor R\rfloor}^{\infty}
		(1+m^{n+4})e^{-cm^2/2}.
	\end{aligned}
	\]
	The right-hand side is independent of $i$ and tends to zero as
	$R\to\infty$. Thus
	\begin{equation}\label{eq:uniform-tail}
		\lim_{R\to\infty}
		\sup_i
		\int_{\mathbf{M}_i\setminus B_{\widetilde g_i}(y_i,R)}
		(F_i+H_i)\,d\mathrm{vol}_{\tilde{g}_i}
		=0,
	\end{equation}
which is uniform in $i$.
	The Euclidean limit integrands $F_\infty$ and $H_\infty$ also have
	Gaussian decay, and hence
	\[
	\lim_{R\to\infty}
	\int_{\mathbb R^n\setminus B(0,R)}
	(F_\infty+H_\infty)\,d\mathrm{vol}_{g_{\mathbb E}}
	=0.
	\]
	
	Combining the convergence of the cut-off integrals with the uniform
	tail estimate gives the desired global convergence. Let $J_i$ denote either $F_i$ or $H_i$, and let $J_\infty$ denote the corresponding limit integrand.
	Fix $\varepsilon>0$. By the uniform tail estimate and the Gaussian decay of the limit integrand, there exists $R>0$ such that
	\[
	\sup_i
	\int_{\mathbf{M}_i\setminus B_{\widetilde g_i}(y_i,R)}
	J_i\,d\mathrm{vol}_{\tilde{g}_i}
	< \frac{\varepsilon}{3},
	\]
	and
	\[
	\int_{\mathbb R^n\setminus B(0,R)}
	J_\infty\,d\mathrm{vol}_{g_{\mathbb E}}
	< \frac{\varepsilon}{3}.
	\]
	For this fixed $R$, the convergence of the localized integrals implies that there exists $i_0=i_0(R,\varepsilon)$ such that, for every $i\geq i_0$,
	\[
	\left|
	\int_{\mathbf{M}_i}\eta_{i,R}J_i\,d\mathrm{vol}_{\tilde{g}_i}
	- \int_{\mathbb R^n}\eta_{\infty,R}J_\infty\,d\mathrm{vol}_{g_{\mathbb E}}
	\right|
	< \frac{\varepsilon}{3}.
	\]
	Since $0\leq\eta_{i,R}\leq1$ and  $\eta_{i,R}=1$ on $B_{\widetilde g_i}(y_i,R)$, we have
	\[
	\int_{\mathbf{M}_i}(1-\eta_{i,R})J_i\,d\mathrm{vol}_{\tilde{g}_i}
	\leq
	\int_{\mathbf{M}_i\setminus B_{\widetilde g_i}(y_i,R)}
	J_i\,d\mathrm{vol}_{\tilde{g}_i}.
	\]
	Similarly,
	\[
	\int_{\mathbb R^n}(1-\eta_{\infty,R})J_\infty\,d\mathrm{vol}_{g_{\mathbb E}}
	\leq
	\int_{\mathbb R^n\setminus B(0,R)}
	J_\infty\,d\mathrm{vol}_{g_{\mathbb E}}.
	\]
	Therefore, for every $i\geq i_0$,
	\[
	\begin{aligned}
		\left|
		\int_{\mathbf{M}_i}J_i\,d\mathrm{vol}_{\tilde{g}_i}
		- \int_{\mathbb R^n}J_\infty\,d\mathrm{vol}_{g_{\mathbb E}}
		\right|
		&\leq
		\left|
		\int_{\mathbf{M}_i}\eta_{i,R}J_i\,d\mathrm{vol}_{\tilde{g}_i}
		- \int_{\mathbb R^n}\eta_{\infty,R}J_\infty\,d\mathrm{vol}_{g_{\mathbb E}}
		\right|\\
		&\quad+
		\int_{\mathbf{M}_i}(1-\eta_{i,R})J_i\,d\mathrm{vol}_{\tilde{g}_i}+
		\int_{\mathbb R^n}(1-\eta_{\infty,R})J_\infty\,d\mathrm{vol}_{g_{\mathbb E}}\\
		&<
		\frac{\varepsilon}{3}
		+ \frac{\varepsilon}{3}
		+ \frac{\varepsilon}{3}
		= \varepsilon.
	\end{aligned}
	\]
	Since $\varepsilon>0$ is arbitrary, we have
	\[
	\lim_{i\to\infty}
	\int_{\mathbf{M}_i}F_i\,d\mathrm{vol}_{\tilde{g}_i}
	= \int_{\mathbb R^n}F_\infty\,d\mathrm{vol}_{g_{\mathbb E}},
	\]
	and
	\[
	\lim_{i\to\infty}
	\int_{\mathbf{M}_i}H_i\,d\mathrm{vol}_{\tilde{g}_i}
	= \int_{\mathbb R^n}H_\infty\,d\mathrm{vol}_{g_{\mathbb E}}.
	\]
	Passing the limit in \eqref{scaled ineq}, we obtain
	\begin{equation}\label{lim ineq}
		\int_{\mathbb{R}^n}\left|\nabla \ln u_{\infty}(\cdot, 1)\right|^2 u_{\infty}(\cdot, 1) \,d\mathrm{vol}_{g_{\mathbb E}} \geq (2+\delta) \int_{\mathbb{R}^n} \left| \nabla^2 \ln u_{\infty}(\cdot, 1)\right|^2 u_{\infty}(\cdot, 1) \,d\mathrm{vol}_{g_{\mathbb E}}.
	\end{equation}

	For the scaled time variable $\tilde{t} = \frac{t}{t_i} \in (0, 1]$, the scaled heat kernels $\tilde{u}_i(x, \tilde{t}) = G_{\tilde{g}_i}(\cdot, \tilde{t}, y_i)$ at $\tilde{t} = 1$ converge in $C_{loc}^2$ to the standard Gaussian heat kernel
	$$
	u_{\infty}(x, 1) = (4\pi)^{-n/2} e^{-|x|^2/4}
	$$
	on $\mathbb{R}^n$.
	We compute the exact values of both integrands as follows
	$$
	\ln u_{\infty}(x, 1) = -\frac{|x|^2}{4} - \ln (4 \pi)^{\frac{n}{2}},\quad
	\nabla \ln u_{\infty}(x, 1) = -\frac{x}{2}\quad\text{and}\quad
	\nabla^2 \ln u_{\infty}(x, 1) = -\frac{1}{2} I.
	$$
	Therefore,
	$$
	\left|\nabla \ln u_{\infty}\right|^2 = \frac{|x|^2}{4}, \quad
	\left|\nabla^2 \ln u_{\infty}\right|^2 = \sum_{j,k} \left(-\frac{1}{2}\delta_{jk}\right)^2 = \frac{n}{4}.
	$$
	
	Now, substituting these explicit forms back into \eqref{lim ineq}, we obtain

	$$\text{LHS} = \int_{\mathbb{R}^n} \frac{|x|^2}{4} \frac{1}{(4 \pi)^{\frac{n}{2}}} e^{-\frac{|x|^2}{4}} \,d\mathrm{vol}_{g_{\mathbb E}} = \frac{n}{2}
	$$
and
	$$
	\text{RHS} = (2+\delta) \int_{\mathbb{R}^n} \frac{n}{4} u_{\infty} \,d\mathrm{vol}_{g_{\mathbb E}}= (2+\delta)\frac{n}{4} = \frac{n}{2}+\frac{\delta n}{4}.
	$$
	The integral inequality then forces
	$$
	\frac{n}{2} \geq \frac{n}{2}+\frac{\delta n}{4},
	$$
	which implies
	$$
	0 \ge \frac{\delta n}{4}.
	$$
	Since $n \ge 1$ and we assumed $\delta > 0$, this is a contradiction.
\end{proof}
 With the above lemma, we can prove the concavity of entropy in the noncompact setting.
\begin{theorem}\label{noncpt E concave}
	Assume that $(\mathbf{M}^n, g)$ is a complete noncompact Riemannian manifold satisfying
		\begin{itemize}
		\item[(a)] $\left|\nabla^j \operatorname{Rm}\right| \leq K,\quad j=0,1,2,3,$
		\item[(b)] the volume noncollapsing condition
		$\inf\limits_{y\in \mathbf{M}} |B(y, 1)| \geq \kappa> 0.$
	\end{itemize}
	Let $u=G(x, t, y)$ be the heat kernel with pole at $y$. For any given $\delta > 0$, there exists a constant
$$
t_*=t_*\left(n, K, \kappa, \delta\right)\in(0,1]
$$
	such that, if
	\[
	\|\operatorname{Ric}^{-}\|_{L^\infty} \leq \frac{1}{(2+\delta) t_*},
	\]
	then the entropy
	\[
	E(t) := -\int_{\mathbf{M}} u(x,t) \ln u(x,t) \, d\mathrm{vol}_g
	\]
	satisfies $E''(t)\leq 0$ for all $t \in (0, t_*],$ i.e., the entropy is concave in time, implying the Fisher information is nonincreasing.
\end{theorem}
\begin{proof}[Proof of Theorem \ref{noncpt E concave}]
	By the definition \eqref{def:Ric}, we have
	$$
	\operatorname{Ric}(\nabla \ln u, \nabla \ln u) \geq - \|\operatorname{Ric}^-\|_{L^\infty} |\nabla \ln u|^2.
	$$
	Substituting this into the second derivative of $E(t)$, i.e., \eqref{eq:E''t}, yields
	$$
	E''(t) \leq -2 \int_{\mathbf{M}} |\nabla^2 \ln u|^2 u \, d\mathrm{vol}_g + 2 \int_{\mathbf{M}} \|\operatorname{Ric}^-\|_{L^\infty} |\nabla \ln u|^2 u \, d\mathrm{vol}_g.
	$$
	Applying our hypothesis $
	\|\operatorname{Ric}^{-}\|_{L^\infty} \leq \frac{1}{(2+\delta) t_*}
	$, we have
	\begin{equation}\label{Ehk bound}
	E''(t) \leq -2 \int_{\mathbf{M}} |\nabla^2 \ln u|^2 u \, d\mathrm{vol}_g + \frac{2}{(2+\delta) t_*} \int_{\mathbf{M}} |\nabla \ln u|^2 u \, d\mathrm{vol}_g.
	\end{equation}

	We now apply the result of Lemma \ref{Wpoinca th}.
	 For all $t \in (0, t_*]$, the gradient term is controlled by the Hessian term
	$$
	\int_{\mathbf{M}} |\nabla \ln u|^2 u \, d\mathrm{vol}_g \leq (2+\delta) t \int_{\mathbf{M}} |\nabla^2 \ln u|^2 u \, d\mathrm{vol}_g.
	$$
	Substituting this inequality into \eqref{Ehk bound}, we obtain that
	$$
	E''(t) \leq -2 \int_{\mathbf{M}} |\nabla^2 \ln u|^2 u \, d\mathrm{vol}_g + \frac{2}{(2+\delta) t_*} \left( (2+\delta) t \int_{\mathbf{M}} |\nabla^2 \ln u|^2 u \, d\mathrm{vol}_g \right).
	$$
	Hence
	$$
	E''(t) \leq -2 \left( 1 - \frac{t}{t_*} \right) \int_{\mathbf{M}} |\nabla^2 \ln u|^2 u \, d\mathrm{vol}_g.
	$$
	
	Since $t \in (0, t_*]$, we have $\frac{t}{t*} \leq 1$. Therefore, the coefficient $ 1 - \frac{t}{t_*} $ is nonnegative. Because the integral $\int_{\mathbf{M}} |\nabla^2 \ln u|^2 u \, d\mathrm{vol}_g$ is always nonnegative, we conclude
	$$
	E''(t) \leq 0.
	$$
\end{proof}
\section{\bf{Density of Entropy (Hamilton Type Estimates) on Manifolds with Nonconvex Boundaries}}
In this section, we first give some technical
preparations needed for the proof of Theorem~\ref{Th: ham estimate}. We recall the following Harnack inequality from \cite[Theorem~3.1]{Jiaping Wang}, rewritten in the notation of the present paper. More precisely, the constant \(H\) in Wang's theorem is replaced here by \(3L+1\).
	\begin{lemma}(\cite[Theorem~3.1]{Jiaping Wang})\label{lem:Harnack}
		Suppose that $u=u(x, t)$ is a positive solution  of the heat equation \eqref{eq:heat eq} under the assumptions of Theorem~\ref{Th: ham estimate}. Then, for any \(\lambda > (2 + 3L)^2 \) and \( 0 <\theta< \frac{1}{2} \), we have
	\[
	u(x_1, t_1)	\leq u(x_2, t_2 ) \left( \frac{t_2}{t_1} \right)^{C_1}\exp\left\{\frac{\lambda d^2(x_1,x_2)}{4 (t_2- t_1)}+C_2(t_2- t_1) \right\},
	\]
	for all \( x_1, x_2 \in \mathbf M \) and $0< t_1< t_2\leq 1$,
	where
	\[
	C_1 = \frac{n\lambda(\lambda - 1)^2(2 + 3L)^4}{(2 - \theta)(1 - \theta)\bigl(\lambda - (2 + 3L)^2\bigr)^2},
	\]
	\[
	C_2 = \frac{6n(\lambda - 1)(2 + 3L)^7 K}{\bigl(\lambda - (2 + 3L)^2\bigr)^2}
	+ \frac{309 n^2 \lambda^2 (3L+1) (\lambda - 1)(2 + 3L)^{10} }{\bigl(\lambda - (2 + 3L)^2\bigr)^4 R^2 \theta}.
	\]
\end{lemma}
We are now ready to give a
\begin{proof}[Proof of Theorem \ref{Th: ham estimate}]
	For \(x\in \mathbf{M}\), let \(r(x)=d(x,\partial \mathbf{M})\) be the distance from \(x\) to the boundary. As in the corresponding constructions in \cite[p.~964-966]{ Chen} and \cite[p.~384, 386]{Jiaping Wang}, we choose \(R>0\) sufficiently small so that \(r(x)\) is smooth in the collar neighborhood \(\{r(x)<R\}\).
	Define a function $\phi(x) = \psi\left(\frac{r(x)}{R}\right)$
	on $\bf{M}$, where \(\psi(r)\) is a nonnegative $C^2$ function defined on \([0,\infty)\) satisfying, for $H=3L+1$
	\begin{equation*}
		\begin{cases}
			\psi(r)\le H&\quad \text{if}\,\,r \in [0, \frac{1}{2}], \\
			\psi(r)= H&\quad \text{if}\,\,r \in [1, \infty ),
		\end{cases}
	\end{equation*}
	and \[
	\psi(0)=0, \quad 0 \leq \psi'(r) \leq 2H, \quad \psi'(0)=H\quad \text{and}\quad \psi''(r) \geq -4H.
	\]
	
	We prove the gradient estimate by using the layered approach in a time interval. For any $t\in (0, 1],$ there exists a fixed $T\in(0, 1]$ such that $t\in [T/2, T]$. Therefore, it is enough to carry out the proof in the time interval $[T/2, T]$ for any fixed $T\in (0, 1]$. To this end, we consider the function
	\[
	Q(x, t):= \left (t-\dfrac{T}{2}\right )\left(1+\phi(x)\right)\dfrac{|\nabla u|^2}{u} - \alpha u\ln \frac{A_T}{u}
	\]
	on $\mathbf M \times[T/2,T]$, where $A_T
	:=
	e\sup_{\mathbf M\times[T/2,T]}u$, $\alpha>0$ will be chosen sufficiently large below.
Since \(Q\) is continuous on the compact set
\(\mathbf M\times[T/2,T]\), there exists
$(x_0,t_0)\in
 \mathbf M\times[T/2,T]
$
 at which \(Q\) attains its maximum. We may assume that
$
Q(x_0,t_0)>0,
$
since otherwise the desired estimate follows immediately.
At \(t=T/2\), we have
\[
 Q(x,T/2)
 =
-\alpha u(x,T/2)
 \ln\frac{A_T}{u(x,T/2)}
 <0.
\]
 Hence \(t_0>T/2\).

 To establish the gradient estimate, we consider two cases: either $x_0 \in \partial \mathbf M$ or $x_0 \notin \partial \mathbf M$.

	\textbf{Case 1: }$x_0\in\partial \mathbf M$.

	Suppose first that \(x_0\in\partial \mathbf M\), and let $\nu$ be the outward unit normal vector field on $\partial \mathbf M$. Since \(Q(\cdot,t_0)\) attains its maximum over \(\mathbf M\) at \(x_0\), its outward normal derivative satisfies
	\[
	\dfrac{\partial Q}{\partial \nu}(x_0,t_0)\ge0.
	\]
	
	Let $e_1, e_2,..., e_n$ be a local orthonormal frame
	field in a neighborhood of  $x_0$ with $e_n=\nu$ along $\partial \mathbf M$. At $(x_0, t_0)$, we have
	\begin{equation*}
		\begin{split}
			\dfrac{\partial Q}{\partial \nu}
			&=\dfrac{\partial }{\partial \nu}\left [\left (t-\dfrac{T}{2}\right )(1+\phi)\dfrac{|\nabla u|^2}{u} - \alpha u\ln \frac{A_T}{u}\right ]\\
			&=\left (t-\dfrac{T}{2}\right )\left [\left (\dfrac{\partial\phi}{\partial \nu}\right ) \dfrac{|\nabla u|^2}{u}+(1+\phi)\dfrac{\partial }{\partial \nu}\left(\dfrac{|\nabla u|^2}{u}\right)\right ].\\
		\end{split}
	\end{equation*}
	Since $\frac{\partial u}{\partial \nu} = 0$ on $\partial M$, it follows that  $\frac{\partial}{\partial \nu}(\alpha u \ln \frac{A_T}{u})=0$ and $\frac{\partial}{\partial \nu}(u^{-1}) = 0$.
	Since $x_0\in\partial M$, we have $r(x_0)=0, \phi (x_0)=0$ and $\dfrac{\partial r}{\partial \nu} \big|_{x_0}=-1$. Consequently,
	$$
	\frac{\partial\phi}{\partial\nu}(x_0) = \left. \frac{\partial\psi\left(\frac{r(x)}{R}\right)}{\partial \nu} \right|_{x=x_0} = \psi'(0)\cdot\frac{1}{R}\cdot \frac{\partial r(x_0)}{\partial \nu} = -\frac{H}{R}.
	$$
	Next, by a direct computation, one shows that
	\small{
	\[
	\frac{\partial}{\partial\nu} (|\nabla u|^2)
	= 2 \sum_{i=1}^{n} u_i u_{i\nu}
	= 2 \sum_{i=1}^{n-1} u_i u_{i\nu}
	=-2 \mathrm{II}_{\partial \mathrm{\bf{M}}}(\nabla u, \nabla u)
	\leq 2L|\nabla u|^2.
	\]}
	Substituting these estimates into  $\frac{\partial Q}{\partial \nu}$ yields
	\[
	\dfrac{\partial Q}{\partial \nu}
	\leq \left (t-\dfrac{T}{2}\right )\frac{|\nabla u|^2}{u}
	\left(-\frac{H}{R}+2L\right).
	\]
Since \(0<R\le1\), $H=3L+1$, we have
	\[
	-\frac{H}{R}+2L
	\le -(3L+1)+2L<0.
	\]
	Since \(Q(x_0,t_0)>0\), we necessarily have \(|\nabla u(x_0,t_0)|>0\). Together with \(t_0>\dfrac{T}{2}\), this implies
	 $\dfrac{\partial Q}{\partial \nu}(x_0, t_0)< 0$, which contradicts the fact that $
	 \dfrac{\partial Q}{\partial \nu}(x_0,t_0)\ge0.
	 $
	 Consequently, the positive maximum of $Q$ must be attained at a point $(x_0,t_0)\in\mathbf M^\circ \times (T/2,T]$, where $\mathbf M^\circ$ is the interior of $\mathrm{\bf{M}}$. So we come to the:
	
	\text{\bf Case 2: }$x_0 \notin \partial \mathbf M$.
	
	For convenience, we define
	$$\mathcal{L}:=\Delta-\partial_t-2\left\langle  \frac{\nabla\phi}{1+\phi}, \nabla \right\rangle,\quad \tau:=t-\dfrac{T}{2}.
	$$
	We first compute $\mathcal LQ$. A direct computation gives
	\begin{equation}
		\begin{aligned}	
			&\mathcal{L}\left (\tau(1+\phi)\dfrac{|\nabla u|^2}{u}\right )\\
			=&\left(\Delta-\partial_t-2\left\langle \frac{\nabla\phi}{1+\phi},  \nabla\right\rangle \right)\left(\tau(1+\phi)\dfrac{|\nabla u|^2}{u}\right)  \\
			=&\tau\Delta\left [(1+\phi)\dfrac{|\nabla u|^2}{u}\right ]
			- (1+\phi )\dfrac{|\nabla u|^2}{u}-
			\tau \partial_t \dfrac{|\nabla u|^2}{u}-2\tau\dfrac{|\nabla \phi |^2}{1+\phi}\dfrac{|\nabla u|^2}{u}
			-2\tau\langle\nabla \phi , \nabla \dfrac{|\nabla u|^2}{u}\rangle \\		
			=&\tau\Delta \phi\dfrac{|\nabla u|^2}{u}
			+\tau (1+\phi)\Delta\dfrac{|\nabla u|^2}{u}
			+2\tau \langle\nabla \phi , \nabla \dfrac{|\nabla u|^2}{u}\rangle\\
			&-(1+\phi )\dfrac{|\nabla u|^2}{u}
			- \tau (1+\phi )\partial_t \dfrac{|\nabla u|^2}{u}
			-2\tau\frac{|\nabla \phi |^2}{1+\phi}\dfrac{|\nabla u|^2}{u}
			-2\tau \langle\nabla \phi , \nabla \dfrac{|\nabla u|^2}{u}\rangle\\
			=&\tau (1+\phi)( \frac{2}{u}
			\left|
			\nabla^2u-
			\frac{\nabla u\otimes\nabla u}{u}
			\right|^2
			+
			\frac{2}{u}
			\operatorname{Ric}(\nabla u,\nabla u))
			+\dfrac{|\nabla u|^2}{u}\left(\tau\Delta \phi-(1+\phi)-2\tau\frac{|\nabla \phi |^2}{1+\phi} \right)\\
		\geq&\left(\tau\Delta \phi-(1+\phi)-2\tau\frac{|\nabla \phi |^2}{1+\phi}-2K\tau (1+\phi) \right)\dfrac{|\nabla u|^2}{u}.
		\end{aligned}
		\label{step1}
	\end{equation}
On the other hand,
	\begin{equation}
		\begin{aligned}
			&\mathcal{L}\left (\alpha u\ln \frac{A_T}{u}\right )\\
			=&\left(\Delta-\partial_t-2\left\langle \frac{\nabla\phi}{1+\phi}, \nabla\right\rangle \right)\left (\alpha u\ln \frac{A_T}{u}\right) \\
			=& \alpha \Delta u(\ln \frac{A_T}{u}-1)
			-\alpha\dfrac{|\nabla u|^2}{u}
			-\alpha  \partial_t u(\ln \frac{A_T}{u}-1)
			-2\alpha\dfrac{\langle \nabla \phi, \nabla u \rangle}{1+\phi}\ln \frac{A_T}{u}
			+2\alpha\frac{\langle \nabla \phi, \nabla u \rangle}{1+\phi}\\
			\leq&-\alpha\dfrac{|\nabla u|^2}{u}
			+ 2\alpha\dfrac{|\nabla\phi|\cdot|\nabla u|}{1+\phi}\ln \frac{A_T}{u}
			+ 2\alpha\frac{|\nabla\phi|\cdot|\nabla u|}{1+\phi} \ln \frac{A_T}{u} \\
			\leq& -\alpha\dfrac{|\nabla u|^2}{u}
			+ 4\alpha\dfrac{|\nabla\phi|}{1+\phi}\cdot\frac{|\nabla u|}{\sqrt{u}}\cdot\sqrt{u}\ln \frac{A_T}{u} \\
			\leq& -\alpha\dfrac{|\nabla u|^2}{u}+\frac12\alpha\dfrac{|\nabla u|^2}{u}
			+8\alpha\dfrac{|\nabla\phi|^2}{(1+\phi)^2}
			u(\ln \frac{A_T}{u})^2,
		\end{aligned}
		\label{step2}
	\end{equation}
	where we have used
	$\ln(A_T/u)\ge1$ and Young's inequality. Combining \eqref{step1} and \eqref{step2}, we obtain
\[
	\begin{aligned}
	\mathcal{L}
	\left (\tau(1+\phi)\dfrac{|\nabla u|^2}{u}
	- \alpha u\ln\frac{A_T}{u}\right)
	&\geq
	\left[\left(\tau\Delta \phi-(1+\phi)-2\tau\frac{|\nabla \phi |^2}{1+\phi} \right)-2K\tau (1+\phi)\right]\dfrac{|\nabla u|^2}{u}\\
	&\quad+\frac{\alpha}{2}\dfrac{|\nabla u|^2}{u}-8\alpha\dfrac{|\nabla\phi|^2}{(1+\phi)^2}
	u(\ln \frac{A_T}{u})^2.
\end{aligned}
\]
	For convenience, set
	\[
	\eta=\eta(x, t)=\frac{\alpha}{2}+\tau\Delta \phi-(1+\phi)-2\tau\frac{|\nabla \phi |^2}{1+\phi}-2K\tau (1+\phi).
	\]
By the construction of the cut-off function, we have
\[0\le \phi\le H,\qquad |\nabla\phi|\le\frac{2H}{R}.
\]
Moreover, the  Laplacian comparison theorem for the boundary distance function \(r(x)=d(x,\partial\mathbf M)\), together with the choice \(H=3L+1\), gives
\[\Delta r\ge -(n-1)(3L+1)\qquad\text{on } 0<r< R,\]
see \cite[p.~384]{Jiaping Wang} and the discussion on the choice of \(R\) in \cite[p.~386]{Jiaping Wang}. See also \cite[p.~964-966]{Chen}, for the analogous construction and the corresponding choice of the boundary collar parameter. Therefore, since
\[
\phi(x)=\psi\left(\frac{r(x)}{R}\right),
\]
and
\[
0\le \psi'\le 2H,
\qquad
\psi''\ge -4H,
\]
we obtain
\[
\begin{aligned}
	\Delta\phi
	&=
	\frac{1}{R^2}\psi''\left(\frac{r}{R}\right)
	+\frac{1}{R}\psi'\left(\frac{r}{R}\right)\Delta r\\
	&\ge
	-\frac{4H}{R^2}
	-\frac{2(n-1)H(3L+1)}{R}.
\end{aligned}
\]
Since $0\le\tau\le T/2\leq1/2$, it follows that
\[
 \begin{aligned}
 	\tau\Delta\phi-(1+\phi)
	 -2\tau\frac{|\nabla\phi|^2}{1+\phi}&\ge-(1+H)
-\frac{T(n-1)H(3L+1)}{R}
-\frac{2TH}{R^2}-\frac{4TH^2}{R^2}\\
&\ge-(1+H)
-\frac{(n-1)H(3L+1)}{R}
-\frac{2H(1+2H)}{R^2}.
	\end{aligned}
\]
Furthermore, combining with
\[
-2K\tau(1+\phi)
\ge - K(1+H),
\]
we deduce
\begin{equation}\label{ineq:eta bd}
\eta
\ge
\frac{\alpha}{2}
-K(1+H)
-(1+H)
-\frac{(n-1)H(3L+1)}{R}
-\frac{2H(1+2H)}{R^2}.
\end{equation}
	With this notation, it follows that
	\[
	\mathcal{L}Q\geq \eta\dfrac{|\nabla u|^2}{u} - 8\alpha\dfrac{|\nabla\phi|^2}{(1+\phi)^2}u(\ln \frac{A_T}{u})^2.
	\]
For later use, we choose $\alpha>0$ sufficiently large such that
$\eta>0$. More precisely, it suffices to take

	\[
	\alpha>
	2\left[
	 K(1+H)
	+(1+H)
	+\frac{(n-1)H(3L+1)}{R}
	+\frac{2H(1+2H)}{R^2}
	\right].
\]

Since \((x_0,t_0)\in \mathbf M^\circ \times (T/2,T]\) is a maximum point of \(Q\),
at this point we have
	\[
	\nabla Q(x_0,t_0)=0,\qquad
	\Delta Q(x_0,t_0)\le0,\qquad
	\partial_tQ(x_0,t_0)\ge0.
	\]
Hence
	\[
	\mathcal{L}Q(x_0,t_0)= \left(\Delta-\partial_t-2\left\langle \frac{\nabla\phi}{1+\phi}, \nabla\right\rangle \right)Q\leq 0.\]
	Since \(Q(x_0,t_0)>0\), i.e.,
\[(t_0-\dfrac{T}{2})(1+\phi)\dfrac{|\nabla u|^2}{u} - \alpha u\ln \frac{A_T}{u}> 0,
	\]
	then \[
	8\alpha\dfrac{|\nabla\phi|^2}{(1+\phi)^2}u(\ln \frac{A_T}{u})^2\geq \eta\dfrac{|\nabla u|^2}{u}> \frac{\eta}{\tau_0(1+\phi)}\alpha u\ln \frac{A_T}{u},
	\]
	where $\tau_0=t_0-\dfrac{T}{2}$.
	Moreover, we find
	\[
	\frac{\eta}{\tau_0(1+\phi)}\alpha u\ln \frac{A_T}{u}-8\alpha\dfrac{|\nabla\phi|^2}{(1+\phi)^2}u(\ln \frac{A_T}{u})^2=\frac{\alpha u}{1+\phi}\ln\frac {A_T}{u}\left( \frac{\eta}{\tau_0}
	-
	8\dfrac{|\nabla\phi|^2}{(1+\phi)}\ln\frac {A_T}{u}\right) < 0.
	\]
	Since
	\[
	u>0
	\qquad\text{and}\qquad
	\ln\frac {A_T}{u}\ge1,
	\]
	we may divide the above by
	\[
	\frac{\alpha u}{1+\phi}\ln\frac{A_T}{u}
	\]
	to obtain
	\[
	\frac{\eta}{\tau_0}
	-
	8\frac{|\nabla\phi|^2}{1+\phi}
	\ln\frac{A_T}{u}
	<0.
	\]
	Define
	\[
	\beta
	:=
	8\sup_{\mathbf M}
	\frac{|\nabla\phi|^2}{1+\phi}\leq\frac{32H^2}{R^2}.
	\]
	Then
	\begin{equation}\label{ineq:contra}
	\frac{\eta}{\tau_0}
	-
	\beta\ln\frac {A_T}{u}
	<0
\end{equation}
	at \((x_0,t_0)\).
	
	We now derive a contradiction at the positive maximum point $(x_0,t_0)$.
	Set
	\[
	M_T
	:=
	\sup_{\mathbf M\times[T/2,T]}u,
	\]
	so that $A_T=eM_T$. By the parabolic maximum principle for the Neumann heat equation,
	 \[
	\max_{\mathbf M}u(\cdot,s)
	\le
	\max_{\mathbf M}u(\cdot,T/4),
	\qquad T/2 \le s \le T.
	\]
Choose $y\in \mathbf M$, such that $u\left(y,\frac T4\right)=\max_{\mathbf M}u(\cdot,T/4)$.
	
	By Lemma~\ref{lem:Harnack}, applied with
	\[
	(x_1,t_1)=\left(y,\frac T4\right),
	\qquad
	(x_2,t_2)=(x,t),
	\]
	and with
	$
	\theta=\frac14, \lambda=2(2+3L)^2,
	$
	we obtain
	\[
M_T
\le u\left(y,\frac T4\right)\le
u(x,t)
\left(\frac{4t}{T}\right)^{C_1}
\exp\left(
\frac{\lambda d^2(x, y)}{4(\tau+T/4)}
+C_2(\tau+T/4)
\right)
	\]
	for $T/2\leq t\le T$, where
$
	C_1=C_1(n,L), C_2=C_2(n,K,L,R)
$
	are the corresponding constants in Lemma~\ref{lem:Harnack}.
Let
\[
D:=\operatorname{diam}(\mathbf M).
\]
Since $t\le T$ and $d(x, y)\le D$, it follows that
\begin{equation}
\ln\frac{M_T}{u(x,t)}\le
C_1\ln 4
+\frac{\lambda D^2}{4(\tau+T/4)}
+C_2 (\tau+T/4).
\end{equation}
By $A_T=eM_T$, we therefore have
\begin{equation}\label{ineq: lnMu}
	\ln\frac{A_T}{u(x,t)}\le 1+
	C_1\ln 4
	+\frac{\lambda D^2}{4(\tau+T/4)}
	+C_2 (\tau+T/4).
\end{equation}
Recall that $\tau=t-\frac{T}{2}$. Since $T/2\leq t\le T$, $0<T\le1$, we get
\[
0\leq \tau \leq T/2\leq\frac{1}{2}, \quad T/4\leq \tau+T/4\le\frac{3T}{4}\le \frac{3}{4}.
\]
Multiplying \eqref{ineq: lnMu} by \(\tau\), we obtain

\begin{equation}
\begin{aligned}\label{ineq:lnGbound}
		\tau\ln\frac{A_T}{u(x,t)}&\le (1+
	C_1\ln 4)\tau
	+\frac{\tau\lambda D^2}{4(\tau+T/4)}
	+C_2\tau(\tau+T/4)\\
	&\le
\frac{1}{2}(1+C_1\ln 4)
+\lambda D^2
+\frac{3C_2}{8}.
\end{aligned}
\end{equation}
On the other hand, from \eqref{ineq:eta bd},
it suffices to choose $\alpha$ such that
\[
\begin{aligned}
	\alpha>&
	2K(1+H)+2(1+H)+\frac{2(n-1)H(3L+1)}{R}+\frac{4H(1+2H)}{R^2}\\
	&+ \beta\left( 2\lambda D^2+\frac{3C_2}{4}+1+C_1\ln4\right).
\end{aligned}
\]
With this choice,
\[
\eta>
\beta\left( \frac{1}{2}(1+C_1\ln 4)+\lambda D^2
+\frac{3C_2}{8}\right)
 >0.
\]
Combining this inequality with \eqref{ineq:lnGbound}, we obtain
\[
\frac{\eta}{\tau}
-\beta\ln\frac{A_T}{u(x,t)}
>0,\qquad x\in\mathbf M,\quad T/2\leq t\le T.
\]

	In particular, at the interior positive maximum point $(x_0,t_0)$ of $Q$, the above inequality contradicts \eqref{ineq:contra} obtained from the maximum principle. Therefore, $Q$ cannot attain a positive maximum on $\mathbf M\times[T/2,T]$. Hence
	\[
	Q(x,t)\le0
	\]
	for all $(x,t)\in \mathbf M\times[T/2,T]$. That is,
	\[
	\left(t-\frac T2\right)
	(1+\phi)
	\frac{|\nabla u|^2}{u}
	\le
	\alpha u\ln\frac{A_T}{u}.
	\]
	Since $1+\phi\ge1$, $T\in(0,1]$ is arbitrary and $A_T\le A$, we conclude
	\[
	\frac{t|\nabla u|^2}{u}
	\le 2
	\alpha u\ln\frac Au.
	\]
\end{proof}

\section{\bf{Gradient Bounds, Density of Entropy and Fisher Information Bound under Integral Curvature Condition}}
In this section we prove Theorem~\ref{Integral GE}, whose statement is given in the introduction.
\begin{proof}[Proof of Theorem \ref{Integral GE}]
	The main ingredient of the proof is a quotient type auxiliary function adapted to the heat kernel representation. Let $J=J(x,t)>0$ be an auxiliary function to be
	specified below and set
	\[
	F = J \frac{|\nabla u|^2}{w},
	\]
	where $w$ and $u$ are defined in the statement of the theorem. Set $\mathcal{P}= \Delta - \partial_t$. We first derive the evolution equation for $F$. The quotient structure allows us to compare the evolution of $|\nabla u|^2$ directly with that of $w$, while the auxiliary factor $J$ will be chosen to absorb the bad terms arising from the Bochner formula under the integral Ricci curvature assumption.
	
	By a direct computation, we obtain
	\begin{equation}\label{eq:th5 1}
		\mathcal{P}|\nabla u|^2= 2|\nabla^2 u|^2 + 2\,\mathrm{Ric}(\nabla u,\nabla u).
\end{equation}
	\[
	\mathcal{P}w = 0, \qquad \mathcal{P}\!\left(\frac{1}{w}\right) = \frac{2|\nabla w|^2}{w^3}.
	\]
	\begin{equation}\label{eq:th5 3}
		\begin{aligned}
			\mathcal{P}\!\left(\frac{|\nabla u|^2 }{w}\right) &= \frac{\mathcal{P}|\nabla u|^2 }{w} + |\nabla u|^2 \,\mathcal{P}\!\left(\frac{1}{w}\right) + 2 \left\langle \nabla\!\left(\frac{1}{w}\right), \nabla |\nabla u|^2\right\rangle   \\
			&= \frac{\mathcal{P} |\nabla u|^2 }{w} + \frac{2|\nabla u|^2 |\nabla w|^2}{w^3} - \frac{2}{w^2}\,\left\langle \nabla w, \nabla |\nabla u|^2\right\rangle.
		\end{aligned}
	\end{equation}
		\begin{equation}\label{eq:th5 4}
	\mathcal{P}F = \frac{|\nabla u|^2}{w}\,\mathcal{P}J + J\,\mathcal{P}\!\left(\frac{|\nabla u|^2}{w}\right) + 2\,\left\langle \nabla J, \nabla\!\left(\frac{|\nabla u|^2}{w} \right)\right\rangle.
		\end{equation}
	Observe that
	
		\begin{equation}\label{eq:th5 5}
	\nabla\!\left(\frac{|\nabla u|^2}{w}\right) = \frac{\nabla |\nabla u|^2}{w}
	- \frac{|\nabla u|^2}{w^2}\nabla w.
		\end{equation}
	Substituting \eqref{eq:th5 3} and \eqref{eq:th5 5} into \eqref{eq:th5 4} yields
\begin{equation}\label{eq:th5 6}
	\begin{split}
		\mathcal{P}F
		&= \frac{|\nabla u|^2}{w}\,\mathcal{P}J
		+ J\left( \frac{\mathcal{P}|\nabla u|^2}{w}
		+ \frac{2|\nabla u|^2|\nabla w|^2}{w^3}
		- \frac{2}{w^2}\left\langle \nabla w, \nabla |\nabla u|^2\right\rangle  \right) \\
		&\quad + 2\left\langle
		\nabla J,
		\frac{\nabla |\nabla u|^2}{w}
		-\frac{|\nabla u|^2}{w^2}\nabla w
		\right\rangle.
	\end{split}
\end{equation}
Substituting \eqref{eq:th5 1} into \eqref{eq:th5 6}, we obtain
	\begin{equation}
		\begin{aligned}
			\mathcal{P}F &= \frac{|\nabla u|^2}{w}\,\mathcal{P}J + \frac{2J}{w}|\nabla^2 u|^2 + \frac{2J}{w}\,\mathrm{Ric}(\nabla u,\nabla u) \\
			&+ \frac{2J|\nabla u|^2|\nabla w|^2}{w^3} - \frac{2J}{w^2}\left\langle \nabla w, \nabla |\nabla u|^2\right\rangle  + \frac{2}{w}\left\langle \nabla J, \nabla |\nabla u|^2\right\rangle  - \frac{2|\nabla u|^2}{w^2}\left\langle \nabla J, \nabla w\right\rangle  .
		\end{aligned}\label{eq:th5 7}
	\end{equation}
By the definition of \(F\), we have
\[
\nabla F=(\frac{|\nabla u|^2}{w})\nabla J+J \nabla\left(\frac{|\nabla u|^2}{w}\right) .
\]
	Substituting \eqref{eq:th5 5} into this identity yields
	\[
	\nabla F = (\frac{|\nabla u|^2}{w})\nabla J + J \left (\frac{\nabla |\nabla u|^2}{w}
	- \frac{|\nabla u|^2}{w^2}\nabla w\right).
	\]
	Hence,
	\[
	2\left\langle \frac{\nabla w}{w}, \nabla F\right\rangle  =
	\frac{2|\nabla u|^2}{w^2}\left\langle \nabla J, \nabla w\right\rangle
	+\frac{2J}{w^2}\left\langle \nabla w, \nabla |\nabla u|^2\right\rangle
	-\frac{2J|\nabla u|^2|\nabla w|^2}{w^3}.
	\]
	Combining this with \eqref{eq:th5 7}, we obtain
		\begin{equation}\label{eq:th5 8}
	\mathcal{P}F+ 2\left\langle \frac{\nabla w}{w}, \nabla F\right\rangle  = \frac{|\nabla u|^2}{w}\,\mathcal{P}J + \frac{2J}{w}|\nabla^2 u|^2 + \frac{2J}{w}\,\mathrm{Ric}(\nabla u,\nabla u)+\frac{2}{w}\left\langle \nabla J, \nabla |\nabla u|^2\right\rangle .
\end{equation}
	
	Set $V = |\mathrm{Ric}^-|$ so that $\mathrm{Ric}(\nabla u,\nabla u) \ge -V |\nabla u|^2.$ Then
	\[
	\frac{2J}{w}\,\mathrm{Ric}(\nabla u,\nabla u) \ge -\frac{2J}{w} V |\nabla u|^2 = -2VJ\frac{|\nabla u|^2}{w} = -2VF.
	\]
	Applying Young's inequality, we get
	\[
	\frac{2}{w}\left\langle \nabla J, \nabla |\nabla u|^2\right\rangle  \ge \frac{1}{w}\left( -\frac{4}{5}\delta J|\nabla^2 u|^2
	- \frac{5}{\delta }\frac{|\nabla J|^2}{J}|\nabla u|^2\right),
	\]
	where \(\delta\) is an arbitrary positive constant.
	
	Substituting these estimates into the previous identity \eqref{eq:th5 8} yields
	\[
	\mathcal{P}F+ 2\left\langle \frac{\nabla w}{w}, \nabla F\right\rangle \geq
	\frac{J}{w}(2- \frac{4}{5}\delta)|\nabla^2 u|^2
	+\frac{|\nabla u|^2}{w}\,\mathcal{P}J
	-2VJ\frac{|\nabla u|^2}{w}
	- \frac{5}{\delta}\frac{|\nabla J|^2}{J}\frac{|\nabla u|^2}{w}.
	\]
	For any $0< \delta \leq \frac{5}{2}$, following the method of \cite{Zhang and Zhu}, let $J$ be a function satisfying the nonlinear parabolic equation
	\begin{equation*}
		\begin{cases}
			\partial_t J = \Delta J - 2V J - \frac{5}{\delta} \frac{|\nabla J|^2}{J}, & (x,t)\in M\times(0,\infty),\\[4pt]
			J(x,0)=1.
		\end{cases}
	\end{equation*}
	We have
	\[
	\mathcal{P}F+ 2\left\langle \frac{\nabla w}{w}, \nabla F\right\rangle  \geq 0.
	\]
	
Without loss of generality we can and do assume $w>\epsilon>0$ and let $\epsilon \to 0$ in the sequel. The maximum principle implies
\[
F(x,t)\le1,
\qquad (x,t)\in\mathbf M\times[0,T].
\]

Using the lower bound
\[
J(x,t)\ge\underline J(t)>0,
\]
we finally obtain
\begin{equation}\label{Integral GE1}
|\nabla u(x,t)|^2
\le
\frac{1}{J(x,t)}\,w(x,t)
\le
\frac{1}{\underline J(t)}\,w(x,t),
\end{equation}
which is the desired estimate.
\end{proof}
\begin{remark}
		When $\operatorname{Ric}\ge0$, one has $V=|\operatorname{Ric}^-|\equiv0$ and the
	equation for $J$ admits the solution $J\equiv1$.  In this case,
	the preceding gradient estimate reduces to the sharp semigroup gradient estimate
	\[
	|\nabla P_tf|^2\le P_t(|\nabla f|^2)
	\]
	for any $f \in C^{\infty}(\mathbf{M})$, where $P_t$ denotes the heat semigroup.
\end{remark}

\begin{remark}
		Let
	\[
	B:=a-1=5\delta^{-1}-1>0,
	\qquad
	q:=\frac{2p}{2p-n}.
	\]
	Under condition (a) of \cite[Theorem 1.1]{Zhang and Zhu}, the auxiliary function $J$ satisfies
	\[
	\underline J(t)\le J(x,t)\le 1,
	\]
	where
	\[
	\underline J(t)
	=
	2^{-1/B}
	\exp\!\left(
	-B^{\frac{n}{2p-n}}
	\bigl[4(\sigma\hat C(t))^{1/p}\bigr]^q\,t
	\right),
	\]
	and $\hat{C}(t)$ is the increasing function appearing in the Gaussian heat kernel estimate.
	Consequently,
	\[
	1\le \frac1{J(x,t)}
	\le
	\frac1{\underline J(t)}
	=
	2^{1/B}
	\exp\!\left(
	B^{\frac{n}{2p-n}}
	\bigl[4(\sigma\hat C(t))^{1/p}\bigr]^q\,t
	\right).
	\]
		The behavior of this bound is different in the short and
	long time.  Indeed, under the integral Ricci curvature
	bound and the noncollapsing assumption, the Gaussian heat kernel
	estimate can be taken with
	\[
	\hat C(t)\le \hat C_1,
	\qquad 0<t\le1,
	\]
	for some constant $\hat C_1>0$.  Hence, for $0<t\le1$,
	\[
	\frac1{\underline J(t)}
	\le
	2^{1/B}
	\exp\!\left(\tilde{C}t\right),
	\qquad
	\tilde{C}=
	B^{\frac{n}{2p-n}}
	\bigl[4(\sigma\hat C_1)^{1/p}\bigr]^q.
	\]
	Thus the global lower bound for $J$ remains uniformly positive
	on every bounded short time interval.
	
	A sharper short time estimate follows from the fixed point
	argument used in the construction of $J$.  More precisely, under
	the assumption $|\operatorname{Ric}^-|\in L^p$, $p>n/2$, one has
	\begin{equation}\label{inte:heat kernel}
	\int_0^t\int_M
	G(x,t;y,s)|\operatorname{Ric}^-|(y)\,d\mathrm{vol}_g(y)\,ds
	\le C_0t^{\,1-\frac{n}{2p}},
	\end{equation}
	where \(C_0>0\) depends only on \(n,p\) and the constants in the
	Gaussian heat kernel and volume estimates, as well as on
	\(\|\operatorname{Ric}^-\|_{L^p}\). In particular,
	\[
	C_0=C(n,p,\rho,\sigma).
	\]
	 Set $h=J^{-B}$. By the Duhamel formula,
	 together with \eqref{inte:heat kernel}, we obtain that for every $\eta\in(0,1)$,
	\[
	1\le h(x,t)\le1+\eta
	\]
	on a
	sufficiently short time interval $[0,T_\eta]$,
	and therefore
	\[
	(1+\eta)^{-1/B}\le J(x,t)\le1.
	\]
	In particular, the auxiliary function \(J(x,t)\) converges uniformly to its initial value \(1\) as \(t\to0^+\), and the corresponding short time gradient estimate becomes asymptotically sharp.
	
	On the other hand, for large time the function $\hat C(t)$ in
	the global Gaussian upper bound is only known in general to be
	increasing and to satisfy $\hat C(t)\to\infty$ as $t\to\infty$.
	Accordingly, the explicit quantity $1/\underline J(t)$ may grow
	rapidly for large $t$.
	
\end{remark}

As a direct and significant application of this gradient estimate, we are able to control the evolution of the Fisher information. Utilizing the estimate above, we deduce the following Fisher information bound. We now turn to the proof of Theorem~\ref{th:Fisher info}, which is inspired by \cite[Theorem 2]{CFM}.

\begin{proof}[Proof of Theorem \ref{th:Fisher info}]
	The Fisher information is defined by
	\begin{equation*}
			\mathcal{I}(t)=\int_{\mathbf{M}}\frac{|\nabla_x u(x,t)|^{2}}{u(x,t)}\,d\mathrm{vol}_g(x).
	\end{equation*}
Since
\[
\nabla u=u\nabla\ln u,
\]
we can rewrite
\begin{equation}\label{eq:It}
\begin{aligned}
	\mathcal I(t)
	&=
	-\int_{\mathbf M}
	|\nabla_x\ln u(x,t)|^2u(x,t)
	\,d\mathrm{vol}_g(x)\\
	&\quad+
	2\int_{\mathbf M}
	\langle\nabla_x\ln u(x,t),\nabla_xu(x,t)\rangle
	\,d\mathrm{vol}_g(x).
\end{aligned}
\end{equation}

	Given an initial density $u_0(x)>0$, we define the finite measure $d\mu(x) =u_0(x)\, d\mathrm{vol}_g(x)$. The density evolved under the heat equation is given by
	$$
	u_t(x) = u(x,t) := \int_{\mathbf{M}} G(x, t, y) \, u_0(y) \, d\mathrm{vol}_g(y),
	$$
	so that
	$d\mu_t=u(x, t)\, d\mathrm{vol}_g(x)$.
	
	Using the symmetry of heat kernel,
	$G(x, t, y) = G(y, t, x)$, together with Fubini's theorem, we obtain
	\begin{equation}\label{term1}
	\begin{aligned}
		&-\int_{\mathbf{M}}|\nabla_x\ln u(x,t)|^{2}u(x,t)\,d\mathrm{vol}_g(x) \\
		= &-\int_{\mathbf{M}}|\nabla_x\ln u(x,t)|^{2}
		\biggl(\int_{\mathbf{M}}G(x, t, y)u_0(y)\,d\mathrm{vol}_g(y)\biggr)d\mathrm{vol}_g(x) \\
		= &-\int_{\mathbf{M}}
		\biggl(\int_{\mathbf{M}}G(y, t, x)|\nabla_x\ln u(x,t)|^{2}\,d\mathrm{vol}_g(x)\biggr)
		u_0(y)\,d\mathrm{vol}_g(y) \\
		= &-\int_{\mathbf{M}}
		\biggl(\int_{\mathbf{M}}G(y, t, x)|\nabla_x\ln u(x,t)|^{2}\,d\mathrm{vol}_g(x)\biggr)
		d\mu(y).
	\end{aligned}
	\end{equation}
	For the second term in \eqref{eq:It}, integration by parts yields
	\[
	2\int_{\mathbf{M}}\langle \nabla_x\ln u(x,t), \nabla_x u(x,t)\rangle \,d\mathrm{vol}_g(x)
	= -2\int_{\mathbf{M}}\Delta_x\bigl(\ln u(x,t)\bigr)\,u(x,t)\,d\mathrm{vol}_g(x).
	\]
	Substituting the heat kernel representation of $u(x,t)$ and applying
	Fubini's theorem once again, we find
	\begin{equation}\label{term2-1}
	\begin{aligned}
		&-2\int_{\mathbf{M}}\Delta_x\bigl(\ln u(x,t)\bigr)u(x,t)\,d\mathrm{vol}_g(x) \\
		= &-2\int_{\mathbf{M}}\Delta_x\bigl(\ln u(x,t)\bigr)
		\biggl(\int_{\mathbf{M}}G(x, t, y)u_0(y)\,d\mathrm{vol}_g(y)\biggr)d\mathrm{vol}_g(x) \\
		= &-2\int_{\mathbf{M}}
		\biggl(\int_{\mathbf{M}}G(y, t, x)\,\Delta_x\bigl(\ln u(x,t)\bigr)\,d\mathrm{vol}_g(x)\biggr)
		d\mu(y).
	\end{aligned}
	\end{equation}
	By integration by parts, the symmetry of the heat kernel, and the heat equation, we have
	\[
	\begin{aligned}
		\int_MG(y,t,x)\Delta_x\ln u(x,t)\,d\mathrm{vol}_g(x)
		&=
		\int_M\Delta_xG(y,t,x)\ln u(x,t)\,d\mathrm{vol}_g(x)\\
		&=
		\int_M\Delta_yG(x,t,y)\ln u(x,t)\,d\mathrm{vol}_g(x)\\
		&=
		\Delta_y\int_MG(y,t,x)\ln u(x,t)\,d\mathrm{vol}_g(x).
	\end{aligned}
	\]
	Substituting this identity into \eqref{term2-1} and integrating
	by parts with respect to $y$, we obtain
		\begin{equation}\label{term2}
	\begin{aligned}
		& -2\int_{\mathbf{M}}
		\Delta_y\Bigl(\int_{\mathbf{M}}G(y, t, x)\ln u(x,t)\,d\mathrm{vol}_g(x)\Bigr)
		u_0(y)\,d\mathrm{vol}_g(y) \\
		= &2\int_{\mathbf{M}}
		\left\langle \nabla_y\Bigl(\int_{\mathbf{M}}G(y, t, x)\ln u(x,t)\,d\mathrm{vol}_g(x)\Bigr)
		, \nabla_y u_0(y)\right\rangle \,d\mathrm{vol}_g(y) \\
		= &2\int_{\mathbf{M}}\left\langle
		\nabla_y\Bigl(\int_{\mathbf{M}}G(y, t, x)\ln u(x,t)\,d\mathrm{vol}_g(x)\Bigr)
		, \nabla_y\ln u_0(y)\right\rangle \, d\mu(y).
	\end{aligned}
	\end{equation}
	Combining \eqref{term1} with \eqref{term2}, we have
	\[
	\begin{aligned}
		\mathcal{I}(t) =
		&-\int_{\mathbf{M}}
		\biggl(\int_{\mathbf{M}}G(y, t, x)|\nabla_x\ln u(x,t)|^{2}\,d\mathrm{vol}_g(x)\biggr)
		d\mu(y) \\
		&+ 2\int_{\mathbf{M}}
		\left\langle \nabla_y\Bigl(\int_{\mathbf{M}}G(y, t, x)\ln u(x,t)\,d\mathrm{vol}_g(x)\Bigr)
		, \nabla_y\ln u_0(y)\right\rangle \,d\mu(y).
	\end{aligned}
	\]
	Using Young's inequality, for any $ \varepsilon > 0$, we obtain
	\begin{equation}\label{last ineq}
		\begin{aligned}
			\mathcal{I}(t) &\leq -\int_{\mathbf{M}}
			\biggl(\int_{\mathbf{M}}G(y, t, x)|\nabla_x\ln u(x,t)|^{2}\,d\mathrm{vol}_g(x)\biggr)
			d\mu(y) \\
			&\quad + \varepsilon \int_{\mathbf{M}} \left |\nabla_y\Bigl(\int_{\mathbf{M}}G(y, t, x)\ln u(x,t)\,d\mathrm{vol}_g(x)\Bigr)\right|^2d\mu(y) \\
			&\quad + \frac{1}{\varepsilon} \int_{\mathbf{M}} |\nabla_y\ln u_0(y)|^2 \, d\mu,
		\end{aligned}
	\end{equation}
	where $d\mu:=u_0(y)\,d\mathrm{vol}_g.$
	For a fixed $t>0$, we apply Theorem~\ref{Integral GE} to the spatial function $f= \ln u(x, t)$. Taking $\varepsilon= \underline{J}(t)$, we obtain
	\[
	\varepsilon\left|\nabla_y \int_{\mathbf{M}} G(y, t, x) \,\ln u(x,t) \, d\mathrm{vol}_g(x)\right|^2
	\le  \int_{\mathbf{M}} G(y, t, x)|\nabla_x\ln u(x,t)|^{2} \, d\mathrm{vol}_g(x).
	\]
Therefore, the first two terms in \eqref{last ineq} cancel, and we
conclude that
	\[
	\mathcal{I}(t) \leq \frac{1}{\underline{J}(t)}\,\mathcal{I}(0).
	\]
	
	This completes the proof.
\end{proof}
\begin{remark}
	Theorem~\ref{Integral GE} is established for closed manifolds,
	and the gradient estimate used in the proof above is based on
	this setting. Therefore, we restrict our
	discussion to the closed manifold case in the present proof.
	
	We remark that the computation of the Fisher information bound
	itself remains valid for the Neumann heat equation on compact
	manifolds with boundary. Indeed, under the Neumann boundary
	condition
	\[
	\partial_\nu u=0
	\quad\text{on }\partial\mathbf M,
	\]
	one also has
	\[
	\partial_\nu\ln u
	=
	\frac{\partial_\nu u}{u}=0.
	\]
	Hence, the boundary terms arising from the integration by parts
	vanish, and the entropy calculation remains valid in the Neumann
	setting. Extending the present Fisher information estimate to this
	case would additionally require the corresponding Neumann version
	of Theorem~\ref{Integral GE}.
\end{remark}

\section*{\bf{Acknowledgement}} The authors wish to thank Professors Lei Ni and Yu Zheng for helpful discussions and encouragement. We should also thank Dr. Xiaohan Cai for telling us that the coefficient in Theorem \ref{th: concave1} (a) in the first version of the paper on arXiv can be improved by a result in their paper \cite{CW}. During the exploratory stage of this work, ChatGPT suggested the candidate counterexample appearing in Remark \ref{counter eg}, assisted in constructing the example showing that inequality \eqref{ln poin1} in Lemma \ref{Lem: poinca th1} is sharp and provided key references. All mathematical arguments were independently verified by the authors.

The first and second authors gratefully acknowledge the support from the National Natural Science Foundation of China (No.~12271163), Key Laboratory of MEA (Ministry of Education), the Science and Technology Commission of Shanghai Municipality (No.~22DZ2229014), and Shanghai Key Laboratory of PMMP (No.~18DZ2271000).


\end{document}